\documentclass[a4paper,10pt]{article}

\usepackage[utf8]{inputenc}
\usepackage[english]{babel}
\usepackage{amsmath,amssymb,amsthm,,amsbsy,bbm} 
\usepackage{hyperref,xcolor,fullpage}

\makeatletter \let\@fnsymbol\@arabic \makeatother 
\allowdisplaybreaks 
\usepackage[numbers]{natbib} 

\newtheorem{theorem}{Theorem}[section]
\newtheorem{lemma}[theorem]{Lemma}
\newtheorem{definition}[theorem]{Definition}

\newtheorem{corollary}[theorem]{Corollary}
\newtheorem{question}[theorem]{Question}
\newtheorem{conjecture}[theorem]{Conjecture}

\newtheorem{proposition}[theorem]{Proposition}
\newtheorem{remark}[theorem]{Remark}

\makeatletter
\let\save@mathaccent\mathaccent
\newcommand*\if@single[3]{
	\setbox0\hbox{${\mathaccent"0362{#1}}^H$}%
	\setbox2\hbox{${\mathaccent"0362{\kern0pt#1}}^H$}%
	\ifdim\ht0=\ht2 #3\else #2\fi }
\newcommand*\rel@kern[1]{\kern#1\dimexpr\macc@kerna}
\newcommand*\widebar[1]{\@ifnextchar^{{\wide@bar{#1}{0}}}{\wide@bar{#1}{1}}}
\newcommand*\wide@bar[2]{\if@single{#1}{\wide@bar@{#1}{#2}{1}}{\wide@bar@{#1}{#2}{2}}}
\newcommand*\wide@bar@[3]{
	\begingroup
	\def\mathaccent##1##2{
		\let\mathaccent\save@mathaccent
		\if#32 \let\macc@nucleus\first@char \fi
		\setbox\z@\hbox{$\macc@style{\macc@nucleus}_{}$}
		\setbox\tw@\hbox{$\macc@style{\macc@nucleus}{}_{}$}
		\dimen@\wd\tw@ \advance\dimen@-\wd\z@ \divide\dimen@ 3 \@tempdima\wd\tw@ \advance\@tempdima-\scriptspace \divide\@tempdima 10 \advance\dimen@-\@tempdima \ifdim\dimen@>\z@ \dimen@0pt \fi \rel@kern{0.6}\kern-\dimen@
		\if#31 \overline{\rel@kern{-0.6}\kern\dimen@\macc@nucleus\rel@kern{0.4}\kern\dimen@} \advance\dimen@0.4\dimexpr\macc@kerna \let\final@kern#2 \ifdim\dimen@<\z@ \let\final@kern1 \fi
		\if \final@kern1 \kern-\dimen@ \fi
		\else \overline{\rel@kern{-0.6}\kern\dimen@#1} \fi }
	\macc@depth\@ne	\let\math@bgroup\@empty \let\math@egroup\macc@set@skewchar 	\mathsurround\z@ \frozen@everymath{\mathgroup\macc@group\relax} 	 \macc@set@skewchar\relax \let\mathaccentV\macc@nested@a	\if#31 \macc@nested@a\relax111{#1} \else \def\gobble@till@marker##1\endmarker{} \futurelet\first@char\gobble@till@marker#1\endmarker \ifcat\noexpand\first@char A\else \def\first@char{} \fi \macc@nested@a\relax111{\first@char} \fi
	\endgroup }
\makeatother

\newcommand{\PP}[1]{\mathbb{P}\left(#1\right)}
\newcommand{\Ind}[1]{\mathbbm{1}_{#1}}
\newcommand{\EE}[1]{\mathbb{E}\!\left(#1\right)}

\newcommand{\median}{\mu_{1/2}}

\newcommand*\bell{\ensuremath{\boldsymbol\ell}}
\newcommand{\ba}{\textbf{a}}
\newcommand{\bb}{\textbf{b}}
\newcommand{\bc}{\textbf{c}}
\newcommand{\bx}{\textbf{x}}
\newcommand{\by}{\textbf{y}}

\newcommand{\bG}{\textbf{G}}

\usepackage{mathrsfs}
\DeclareMathAlphabet{\mathpzc}{OT1}{pzc}{m}{it}
\usepackage{xifthen}
\newcommand{\fp}[1][]{ \ifthenelse{\isempty{#1}}{\mathpzc{p}}{\mathpzc{p}(#1)} }

\newcommand{\mS}{\mathcal{S}}
\newcommand{\mH}{\mathcal{H}}
\newcommand{\mP}{\mathcal{P}}
\newcommand{\mF}{\mathcal{F}}

\newcommand{\mG}{\mathcal{G}}
\newcommand{\mK}{\mathcal{K}}

\newcommand{\mB}{\mathcal{B}}

\newcommand{\NN}{\mathbb{N}}
\newcommand{\ZZ}{\mathbb{Z}}

\newcommand{\rank}[1]{\text{rk} ({#1})}
\newcommand{\Sub}[2]{{#1}[{#2}]}

\newcommand{\density}[1]{d \! \left( {#1} \right)}

\newcommand{\floor}[1]{\lfloor {#1} \rfloor}
\newcommand{\ceil}[1]{\lceil {#1} \rceil}
\newcommand{\bigfloor}[1]{\left \lfloor {#1} \right \rfloor}
\newcommand{\bigceil}[1]{\left \lceil {#1} \right \rceil}

\definecolor{darkred}{cmyk}{.3,.9,.80,.2}

\title{On the optimality of the uniform random strategy} 
\date{}
\author{
	Christopher Kusch \thanks{Freie Universit\"at Berlin, Institut
          f\"ur Mathematik und Informatik, Arnimallee 3, 14195 Berlin,
          Germany and Berlin Mathematical School, Germany. E-mail:
          {\tt c.kusch@gmx.net}. Supported by Berlin Mathematical
          School Phase II scholarship.} \and
	Juanjo Ru{\'e} \thanks{Universitat Polit\`ecnica de Catalunya and Barcelona Graduate School of Mathematics, Department of Mathematics, Edificio Omega, 08034 Barcelona, Spain. E-mail: {\tt juan.jose.rue@upc.edu}. Supported by the Spanish Ministerio de Econom\'{i}a y Competitividad project MTM2014-54745-P and the Mar\'ia de Maetzu research grant MDM-2014-0445.} \and
	Christoph Spiegel \thanks{Universitat Polit\`ecnica de Catalunya and Barcelona Graduate School of Mathematics, Department of Mathematics, Edificio Omega, 08034 Barcelona, Spain. E-mail: {\tt christoph.spiegel@upc.edu}. Supported by the Spanish Ministerio de Econom\'{i}a y Competitividad FPI grant under the project MTM2014-54745-P and the Mar\'ia de Maetzu research grant MDM-2014-0445.} \and
	Tibor Szab{\'o} \thanks{Freie Universit\"at Berlin, Institut f\"ur Mathematik und Informatik, Arnimallee 3, 14195 Berlin, Germany. E-mail: {\tt szabo@math.fu-berlin.de}. Supported by GIF grant G-1347-304.6/2016.}
}

\begin{document}
\maketitle

\begin{abstract}
Biased Maker-Breaker games, introduced by Chv\'atal and Erd\H os, are central to the field of positional games and have deep connections to the theory of random structures. The main questions is to determine the smallest bias needed by Breaker to ensure that Maker ends up with an independent set in a given hypergraph. Here we prove matching general winning criteria for Maker and Breaker when the game hypergraph satisfies certain 'container-type' regularity conditions. This will enable us to answer the main question for hypergraph generalizations of the $H$--building games studied by Bednarska and {\L}uczak~\cite{BL00} as well as a generalization of the van der Waerden games introduced by Beck~\cite{Be81}. We find it remarkable that a purely game-theoretic deterministic approach provides the right order of magnitude for such a wide variety of hypergraphs, while the analogous questions about sparse random discrete structures are usually quite challenging.
\end{abstract}

\section{Introduction}

For a positive integer $k\in \NN$, a \emph{$k$--term arithmetic progression} (or {\em $k$--AP}) is a set of integers that can be written in the form $\{ a, a+d, \ldots , a+(k-1)d \}$ for some $a\in \ZZ$ and $d \in \NN$. The classical theorem of van der Waerden~\cite{vdW27} states that for every $k \in \NN$ there exists an integer $n$ such that any two-colouring of $[n] : = \{1,...,n\}$ contains a monochromatic $k$--AP. The smallest such integer $n$ is called the {\em  (2-color) van der Waerden number} and is denoted by $W(k)$. The determination of $W(k)$ is one of the notorious open problems of combinatorics, with a rich history and connections to many other branches of mathematics.  The best known upper and lower bounds on $W(k)$ are very far from each other: there are several lower bounds of the form $2^{k(1+o(1))}$~\cite{Berlekamp68, ZoltanSzabo}, while the best known upper bound, due to Gowers~\cite{Gow2001} is a tower function of height five.

Beck~\cite{Be81} introduced \emph{van der Waerden games} as the positional games played on the set $[n]$ by two players, Maker and Breaker, who take turns in occupying integers that have previously not been occupied by any of the players, with Maker going first. For an integer $k\in \NN$ Maker wins the van der Waerden game (or {\em $k$--AP game}) if he manages to occupy a $k$--AP, while Breaker wins otherwise (that is, if he occupies at least one integer in every $k$--AP). Van der Waerden's Theorem implies that Breaker cannot win a $k$--AP game on $[W(k)]$ without occupying a $k$--AP for himself. From this, a standard strategy stealing argument (see \cite{Be08}) shows that Breaker cannot have a winning strategy on $[W(k)]$, and consequently Maker has to have one. Beck defined $W^{\star}(k)$ to be the smallest integer $n$, such that Maker has a winning strategy in the $k$--AP game played on $[n]$. By the above, $W^{\star}(k)\leq W(k)$ for every $k\in \NN$. In strong contrast to the enormous gap between the known upper and lower bounds for $W(k)$, Beck established that the van der Waerden game number is single exponential: $W^{\star}(k)= 2^{k(1+o(1))}$.

{\em Biased games} represent a central direction of the field of positional games, with deep connections to the theory of random structures. The notion was first suggested by Chv\'atal and Erd\H{o}s~\cite{CE78} while investigating the connectivity game, hamiltonicity game and triangle-building game played on the edges of the complete graph. Given a hypergraph $\mH = (V(\mH), E(\mH))$ and a positive integer \emph{bias} $q$, we define the \emph{$q$--biased Maker-Breaker game $\bG(\mH;q)$} as the game where \emph{Maker} and \emph{Breaker} take turns occupying previously unoccupied vertices from $V(\mH)$, with Maker going first and occupying one vertex in each round and Breaker occupying up to $q$. Maker wins if his selection completely covers an edge from $\mH$ and Breaker wins otherwise. In other words, Breaker wins if and only if the vertices of Maker form an independent set in $\mH$. Note that by definition the game cannot end in a draw. Given a hypergraph $\mH$, one is interested in determining the \emph{threshold bias $q(\mH)$}, defined to be the smallest integer $q\in \NN$ for which  Breaker has a winning strategy in the $q$--biased game $\bG(\mH;q)$.

The relationship of biased games to random discrete structures originates from the simple observation that if both Maker and Breaker occupy their vertices uniformly at random from the remaining free vertices, then Maker ends up occupying a uniform random set of size $|V(\mH)|/(b+1)$, where $b$ is the bias of Breaker. For exactly what size a uniform random subset will likely be independent (and hence for what bias $b$ will {\tt RandomBreaker} likely win against {\tt RandomMaker}) is a central line of research in the study of random discrete structures. This threshold set size was investigated and is well-known for many important hypergraphs.

Chv\'atal and Erd\H os~\cite{CE78} were mostly concerned about graph games where the vertex set of the game hypergraph ${\cal H}$ is the edge set $E(K_n)$ of the complete graph on $n$ vertices and ${\cal H}$ represents a graph property. For the connectivity game they proved the surprising phenomenon that the threshold bias for the game with ‘clever’ players is of the same order as the ‘likely’ threshold bias in the game with random players. In other words the {\em result} of the clever game and the random game are likely to be the same for almost all biases except an interval of length of smaller order than the value of the threshold bias. This result was later strengthened to establish the equality of the constant factors of the threshold biases and extended also for the Hamiltonicity game~\cite{GebauerSzabo, Krivelevich}.

For the triangle-building game Chv\'atal and Erd\H os resolved the issue of the threshold bias via ad-hoc strategies, and found the above phenomenon very much {\em not} to be true. Namely the {\em clever} ad-hoc Breaker-strategy of Chv\'atal and Erd\H os used a bias much smaller than is needed for {\tt RandomBreaker} in the random game, and won not only against a random,  but also against a clever Maker. The real reason for this and the question for other $H$--building games remained a mystery until its spectacular resolution by Bednarska and \L uczak~\cite{BL00}. They have established a {\tt CleverBreaker} strategy that won the $H$--building game using a much smaller bias than what is needed in the random game and they showed that their bias is optimal up to a constant factor. Writing
	\begin{equation*}
		m_2(H) = \max_{\substack{G \subset H \\ v(G) \geq 3}} \frac{e(G) - 1}{v(G) - 2},
	\end{equation*}
	their result can be stated as follows.
\begin{theorem}[Bednarska-\L uczak]
	For every graph $H$ with at least three non-isolated vertices, the threshold bias of the $H$--building game is of order $n^{1/m_2(H)}$.
\end{theorem}

To gain an intuition for the order of this threshold bias, it is worthwhile to investigate the general lower bound on $q({\cal H})$, which is delivered by the uniform random strategy of Maker. Namely, if Maker occupies a uniformly random free element of $V({\cal H})$ in each round and wins with non-zero probability against a {\tt CleverBreaker} playing with bias $b$, then clearly $b\leq q({\cal H})$. It is important to note though that in this `half-random' scenario {\tt RandomMaker}'s random set of size $|V({\cal H})|/(b+1)$ is not anymore uniformly random, it depends very much on {\tt CleverBreaker}'s strategy. So in analyzing the largest such $b$, it is not anymore relevant how sparse a uniformly random set of size $|V(\mH)|/(b+1)$ is likely to contain a hyperedge. It turns out however that the success of the uniform random strategy of Maker can be salvaged if, for some constant $\varepsilon >0$, the uniform random set of size $\varepsilon |V(\mH)|/(b+1)$ not only is expected to contain a hyperedge, but more resiliently, {\em every $\delta$--fraction of it} is expected to contain a hyperedge, for some $\delta < 1$. In other words, if the random induced subhypergraph of ${\cal H}$ of density $c/(b+1)$ is ‘globally resilient' to the property of being independent (for the precise definitions concerning resilience, see \cite{SudakovVu}). Then, if {\tt RandomMaker} can ensure that, despite {\tt CleverBreaker}'s strategy, he actually occupies at least a $\delta$--fraction of some uniform random subset of that size, he wins. Bednarska and \L uczak~\cite{BL00} managed to implement this plan and couple it with an appropriate Breaker strategy.

The study of the smallest possible such resilience constant $\delta$ in various random hypergraphs was a driving force of research in discrete probability in the past three decades. Eventually it has lead to the determination of the smallest possible constant in many natural scenarios \cite{CG10, Sch12}. In particular it is known in many cases what is the order of the smallest random set size for which the corresponding resilience constant is less than $1$ (a much easier task). This included the $e(H)$--uniform hypergraph ${\cal B}_n(H)$ on $E(K_n)$ where the hypergedges correspond to the copies of a fixed graph $H$ in $K_n$; this is the setup of the Bednarska-\L uczak result.

Building on their ideas, we extend the results of Bednarska and \L uczak to a whole range of other hypergraphs. All our results will follow from two general winning criteria, one for Maker and one for Breaker, which apply for hypergraphs possessing some form of ‘container-type’ regularity conditions, that properly separate the maximum degree of $\ell$--element vertex sets from the average degree. These hypergraphs in particular include the ones corresponding to  Beck's van der Waerden games. More generally, we also obtain tight results for a much broader class of games we call {\em Rado games}, in which Maker's goal is to occupy a solution to an arbitrary given linear system of equations. Finally we also extend the tight results of Bednarska and {\L}uczak~\cite{BL00} to hypergraph-building games for arbitrary fixed uniformity.

It is worthwhile to note that the analogous extension from graphs to hypergraphs represented a significant jump in difficulty for the analogous sparse random problems~\cite{CG10, Sch12}, while here we obtain it for a wider classes of hypergraphs, using a deterministic Breaker strategy. Furthermore, the container method, so effective there, only provides a Maker winning strategy with a $\log$--factor below the optimal bias.

In the remainder of this introduction, we will first give a formal statement of our two winning criteria. This will be followed by the statement of our results about Rado games -- the generalization of Beck's van der Waerden games for linear systems. Lastly, we will describe our results regarding hypergraph-building games.

\subsection{General Winning Criteria}

In order to simplify notation we often identify the hypergraph $\mH$ with its edge set $E(\mH)$. We denote the number of vertices of a hypergraph $\mH$ by $v(\mH)$, the number of edges by $e(\mH)$ and its \emph{density} by $\density{\mH} = e(\mH) / v(\mH)$. Given a subset $S \subseteq V(\mH)$ of vertices, let $\deg(S) = | \{ e\in \mH: S \subset e \}|$. For any integer $\ell \in \NN$ the \emph{maximum $\ell$--degree} is given by $\Delta_{\ell}(\mH) = \text{max} \{\deg(S): S \subseteq V(\mH), |S| = \ell \}$. Note that if $\mH$ is $k$--uniform for some integer $k \in \NN$, then $\Delta_k(\mH) = 1$ and $\Delta_\ell(\mH) = 0$ for all integers $\ell > k$.

Given some sequence of hypergraphs $\mH = (\mH_n)_{n \in \NN}$, the first statement now gives a criterion for a lower bound of the threshold biases of $\bG(\mH_n;q)$.
\begin{theorem} \label{cor:MakerWinCriterion}
	For every $k \geq 2$ the following holds. If $\mH = (\mH_n)_{n \in \NN}$ is a sequence of $k$--uniform hypergraphs that satisfies
	\begin{align*}
		 \text{(M1)} \enspace \Delta_1(\mH_n) = O \big( \density{\mH_n} \big), \qquad \text{(M2)} \enspace \Delta_2 (\mH_n) = o \big( \density{\mH_n} \big) \quad \text{and} \quad \text{(M3)} \enspace \density{\mH_n} = o \big( v(\mH_n)^{k-1} \big)
	\end{align*}
	then the threshold biases of the games played on $\mH_n$ satisfy
	\begin{equation*}
		q(\mH_n) = \Omega \left( \min_{2 \leq \ell \leq k} \left( \frac{\density{\mH}}{\Delta_{\ell}(\mH)} \right)^{\frac{1}{\ell-1}} \right).
	\end{equation*}
\end{theorem}
The proof of Theorem~\ref{cor:MakerWinCriterion} is based on a random strategy for Maker and will be given in Section~\ref{sec:MakerProof}.

The second statement gives a criterion for an upper bound of the threshold biases of $\bG(\mH_n;q)$.
\begin{theorem} \label{cor:BreakerWinCriterion}
	For every $k \geq 2$ the following holds. If $\mH = (\mH_n)_{n \in \NN}$ is a sequence of $k$--uniform hypergraphs such that $v(\mH_n) \to \infty$ and there exists an $\epsilon > 0$ so that for every $2 \leq \ell \leq k-1$ we have
	\begin{align*}
		\Delta_{\ell} (\mH_n)^{\frac{1}{k-\ell}} \, v(\mH_n)^{\epsilon} = O \big( \Delta_1 (\mH_n)^{\frac{1}{k-1}} \big)
	\end{align*}
	then the threshold biases of the games played on $\mH_n$ satisfy
	\begin{equation*}
		q(\mH_n) = O \left( \Delta_1 (\mH_n)^{\frac{1}{k-1}} \right).
	\end{equation*}
\end{theorem}
The proof will be given in Section~\ref{sec:BreakerProof} and constructs an explicit winning strategy for Breaker through multiple applications of the biased Erd\H{o}s-Selfridge Criterion of Beck together with a bias-doubling strategy that mimics a common alteration approach of the probabilistic method.

A joint statement that determines the exact order of the threshold bias for a wide class of games will be stated as a remark in the last section. The approach used both for the proof of Maker's criterion, as well as to devise Breaker's optimal strategy and prove its validity, will follow the general lines of the proofs from~\cite{BL00}.

\subsection{Rado Games}

Arithmetic progressions of length $k$ (or {\em $k$--AP} for short) can be described as the non-trivial solutions of the linear homogenous equation system of $k-2$ equations and $k$ variables
\begin{equation*}
	\begin{pmatrix}
	1 & -2 & 1 &  \\
	& 1 & -2 & 1 \\
	&  & & ... &  & \\
	&  &   & & 1 & -2 & 1 &
	\end{pmatrix} \cdot \bx^T = {\bold 0}^T.
\end{equation*}
In our paper we treat positional games corresponding to an {\em arbitrary} integer matrix $A$. The game is played on the set $[n]$ and Maker's goal is to occupy a solution to $A \cdot \bx^T = \bb^T$ for some given vector $\bb \in \ZZ^r$. We call these games \emph{Rado games}, motivated by the classic result of Rado in Ramsey Theory~\cite{GRS90}. This notion extends van der Waerden games introduced by Beck~\cite{Be81}.

In the context of Rado's Theorem, combinatorial research focused mostly on {\em proper} solutions to the homogeneous case, that is solutions with pairwise distinct entries and $\bb = {\bold 0}$. We state our results in the case when occupying a proper solution is required for Maker to win. The effect that solutions with repeated components of solutions have on the game will be discussed in Subsection~\ref{subsec:repeatedcomponents}.

For an integer-valued matrix $A\in \ZZ^{r\times m}$ and integer-valued vector $\bb \in \ZZ^r$, we denote by
\begin{equation*}
	S(A,\bb) = \{ \bx \in \ZZ^m : A \cdot \bx^T = \bb^T \}
\end{equation*}
the set of \emph{all} integer solutions and let
\begin{equation*}
	S_0(A,\bb) = \{ \bx = (x_1, \dots , x_m) \in S(A,\bb) : x_i \neq x_j \text{ for } i \neq j \}
\end{equation*}
denote the set of all {\em proper} integer solutions. The $m$--uniform hypergraph of the game that accepts only proper solutions from $[n]$ is denoted by
\begin{equation*}
	\mS_0(A,\bb,n) = \big\{ \{ x_1 , \dots , x_m \} : (x_1, \dots , x_m) \in S_0(A,\bb) \cap [n]^m \big\}.
\end{equation*}

Next, we give the definition of some basic properties which largely determine the behaviour of the game. For $Q \subseteq [m]$, we denote by $A^Q$ the matrix obtained from $A$ by keeping only the columns indexed by $Q$. By convention, $A^{\emptyset}$ is the empty matrix of rank $0$.
We then say that a matrix $A \in \ZZ^{r \times m}$ is
\begin{itemize} \setlength\itemsep{0em}
	\item[(i)] \emph{positive} if $S(A,{\bold 0}) \cap \NN^m \neq \emptyset$, that is, there are solutions whose entries lie in the positive integers,
	\item[(ii)] \emph{abundant} if $\rank{A} > 0$ and $\rank{A^Q} = \rank{A}$ for all $Q \subseteq [m]$ satisfying $|Q| \geq m-2$, that is, $A$ has rank strictly greater than $0$ and every submatrix obtained from $A$ by deleting at most two columns must be of the same rank as $A$.
\end{itemize}

The importance of the first notion is easy to justify: if the homogeneous solution space is disjoint from the positive quadrant in which we are playing, then either it is contained in a subspace or the (inhomogeneous) game hypergraph will contain at most a finitely bounded number of winning sets for all $n$. In the former case there must exist at least one $i \in [m]$ such that $x_i = b_i$ for any $\bx = (x_1,\dots,x_m) \in \mS(A,\bb)$. The game would therefore be massively determined by which player occupies the point(s) $b_i$: if there exist several distinct such $b_i$, then Breaker can occupy at least one of them in his first round and hence win the game. If there exists just one such value, Maker occupies it with his first move and the game is reduced to a lower-dimensional one. In the later case, where the game hypergraph contains only a finitely bounded number of winning set, the bias threshold would simply be bounded by some positive constant.

The second definition might initially be somewhat obscure, but we will see that it is of great importance. On the one hand we will show that if a positive matrix is also abundant, then for any vector $\bb \in A(\mathbb{Z}^r)$ the system  $A \cdot \bx^T = \bb^T$ has many proper positive solutions, and hence the corresponding positional game is interesting. On the other hand non-abundant systems turn out to be ‘degenerate’ in some sense and in particular,  Breaker wins the game with a bias of at most $2$.

For readers familiar with the notion of \emph{partition} and \emph{density regular} (or \emph{invariant}) matrices in the homogeneous setting, note that they are trivially positive. See~\cite{Sp16} for an easy proof that they are also abundant..

\medskip

Next, in order to state our main theorem for Rado games, we define a parameter for abundant matrices. (The same parameter was introduced earlier by Rödl and Ruci\'nski~\cite{RR97} for partition regular matrices.) Let $r_Q = \rank{A} - \rank{A^{\widebar{Q}}}$ for any set of column indices $Q \subseteq [m]$ where we set $\rank{A^{\emptyset}} = 0$. The \emph{maximum $1$--density} of an abundant matrix $A \in \ZZ^{r \times m}$ is defined as
\begin{equation*} \label{eq:max2density}
	m_1(A) = \max_{\substack{Q \subseteq [m] \\ 2 \leq |Q|}} \frac{|Q|-1}{|Q|-r_Q-1}.
\end{equation*}
We will later show in Lemma~\ref{lemma:welldefined} that this parameter is indeed well-defined, that is $|Q|-r_Q-1 > 0$ for all $Q \subseteq[m]$ satisfying $|Q| \geq 2$ if $A$ is abundant. Note that this parameter has some clear parallels to the $2$--density of a graph. For more details on the connections to their result and others in the area of random sets and graphs, we refer to the remarks given in Subsection~\ref{subsec:probintuition}.

We refer to the biased Maker-Breaker game played on the hypergraph $\mS_0(A,\bb,n)$ as the \emph{Maker-Breaker $(A,\bb)$--game on $[n]$}. Our main result regarding Rado games states the asymptotic behaviour of the threshold bias of these games when $A$ is abundant.

\begin{theorem} \label{thm:ThresholdGeneralizedVdWGames-Proper}
	For every positive and abundant matrix $A \in \ZZ^{r \times m}$ and vector $\bb \in \ZZ^r$ such that $S(A,\bb) \neq \emptyset$, the threshold bias of the Maker-Breaker $(A,\bb)$--game on $[n]$ satisfies $q(\mS_0(A,\bb,n)) = \Theta \left( n^{1/m_1(A)} \right)$.
\end{theorem}

Let us also state a proposition regarding those positive matrices not covered by the previous result. This will be much easier to prove than the result of Theorem~\ref{thm:ThresholdGeneralizedVdWGames-Proper} and one can see that the simple structure of non-abundant matrices strongly favours Breaker compared to the abundant case.

\begin{proposition} \label{prop:ThresholdEasyGeneralizedVdWGames}
	For positive but non-abundant matrix $A \in \ZZ^{r \times m}$ and vector $\bb \in \ZZ^r$, the threshold bias of the Maker-Breaker $(A,\bb)$--game on $[n]$ satisfies $q(\mS_0(A,\bb,n)) \leq 2$.
\end{proposition}

\subsection{Hypergraph $\mG$--Games}

Given some graph $G$ on at least $3$ non-isolated vertices, one defines its \emph{$2$--density} as
\begin{equation*}
	m_2(G) = \max_{\substack{F \subseteq G \\ v(F) \geq 3  }}\frac{e(F)-1}{v(F)-2}.	
\end{equation*}
As previously mentioned, Bednarska and {\L}uczak~\cite{BL00} showed that the threshold bias in the Maker-Breaker game played on the edge set of $K_n$ where Maker tries to occupy a copy of $G$ satisfies $q({\cal B}_n(G)) = \Theta (n^{1/m_2(G)})$ where ${\cal B}_n(G)$ denotes the $e(G)$--uniform hypergraph of all copies of $G$ in $K_n$.

We consider the following generalization. Given some $r$--uniform hypergraph $\mG$ on at least $r+1$ non-isolated vertices, we define the \emph{$r$--density} of $\mG$ to be
\begin{equation*}
	m_r(\mG) = \max_{\substack{\mF \subseteq \mG \\ v(\mF) \geq r+1}} \frac{e(\mF)-1}{v(\mF)-r}.
\end{equation*}
Note that this is an obvious generalization of the $2$--density of a graph. Furthermore, we call $\mG$ \emph{strictly $r$--balanced}, if $m_r(\mG) > (e(\mF)-1)/(v(\mF)-r)$ for every subhypergraph $\mF$ of $\mG$ on at least $r+1$ vertices.

Let ${\cal B}_n(\mG)$ denote the $|{\cal G}|$--uniform hypergraph of all copies of $\mG$ in the complete $r$--uniform hypergraph $\mK_n^{(r)}$. Using the general winning criteria for Maker and Breaker, we generalize the result of Bednarska and {\L}uczak to the \emph{Maker-Breaker $\mG$--game on $\mK_n^{(r)}$}, that is the game in which Maker tries to occupy a copy of $\mG$ in $\mK_n^{(r)}$.
\begin{theorem} \label{thm:ThresholdHypergraphGames}
	For any integer $r \geq 2$ the following holds. If $\mG$ is an $r$--uniform hypergraph on at least $r+1$ non-isolated vertices, then the threshold bias of the Maker-Breaker $\mG$--game on $\mK_n^{(r)}$ satisfies $q(\mH(\mG,n)) = \Theta \left( n^{1/m_r(\mG)} \right)$.
\end{theorem}

\medskip

\noindent {\bf Outline.} In the remainder of this paper we will first formulate a stronger combinatorial version of Theorem~\ref{cor:MakerWinCriterion}, the general criterion for Maker, and then provide a proof based on a probabilistic result of Janson, {\L}uczak and Ruci\'nski in \emph{Section~\ref{sec:MakerProof}}. This will be followed by a combinatorial formulation of Theorem~\ref{cor:BreakerWinCriterion}, the general criterion for Breaker, which will be proven through a combination of several smaller strategies aimed at avoiding clustering of solutions in \emph{Section~\ref{sec:BreakerProof}}. Subsequently we will prove the two applications Theorem~\ref{thm:ThresholdGeneralizedVdWGames-Proper} and Theorem~\ref{thm:ThresholdHypergraphGames}, our generalization of Beck's van der Waerden games as well as the fixed hypergraph games, through applications of the two general criteria in \emph{Sections~\ref{sec:vdWGames} and~\ref{sec:HypergraphGames}} respectively. Section~\ref{sec:vdWGames} will also include several common minor results regarding the counting of solutions  as well as the proof to Proposition~\ref{prop:ThresholdEasyGeneralizedVdWGames} and a discussion of the effect that solutions with repeated components of solutions have on the game. Lastly, \emph{Section~\ref{sec:conclusion}} will contain some remarks and open questions as well as a reflection on the strong connection between these results and some well-known recent probabilistic statements in extremal combinatorics.

\section{Proof of Theorem~\ref{cor:MakerWinCriterion} -- Maker's Strategy} \label{sec:MakerProof}

We start by stating a strengthening of
Theorem~\ref{cor:MakerWinCriterion} that we will actually prove. In order to do so, we introduce the function
\begin{equation*}
	f(\mH) = \min_{2 \leq \ell \leq k} \left( \frac{\density{\mH}}{\Delta_{\ell}(\mH)} \right)^{\frac{1}{\ell-1}}
\end{equation*}
for any given $k$--uniform hypergraph $\mH$ and note that $1/f(\mH) = \max_{2 \leq \ell \leq k} \left( \Delta_{\ell}(\mH) / \density{\mH} \right)^{1/(\ell-1)}$. The combinatorial winning criterion for Maker now can be stated as follows. We will see how to derive Theorem~\ref{cor:MakerWinCriterion} from it immediately afterwards.
\begin{theorem}[Maker Win Criterion]\label{thm:MakerWinCriterion}
	For every $k \geq 2$ and every positive $c_1 \geq k$ there
        exists $c = c(k,c_1) > 0$ and $\bar{c} = \bar{c}(k,c_1) >
        0$ such that the following holds. If $\mH$ is a $k$--uniform
        hypergraph satisfying
	\begin{align*}
		 \text{(Mi)} \enspace \Delta_1(\mH)\leq c_1 \,
          \density{\mH}, \qquad  \text{(Mii)} \enspace f(\mH) > 1,
          \qquad \text{(Miii)} \enspace \frac{v(\mH)}{f(\mH)} \left(1 - \frac{1}{f(\mH)} \right) \geq \bar{c}
	\end{align*}
	then Maker has a winning strategy in $\bG(\mH;q)$ provided
	\begin{equation} \label{eq:MakerBiasBound}
		q  \leq c \, f(\mH) - 1.
	\end{equation}
\end{theorem}

We start by proving that Theorem~\ref{cor:MakerWinCriterion} is a consequence of this result.

\begin{proof}[Proof of Theorem~\ref{cor:MakerWinCriterion} from Theorem~\ref{thm:MakerWinCriterion}]
	We see that \emph{(M1)} immediately implies Condition~\emph{(Mi)} for $n$ large enough. Now \emph{(M2)} implies $\density{\mH_n} / \Delta_2 (\mH_n) \to \infty$. As $\Delta_{\ell} (\mH_n) \leq \Delta_2 (\mH_n)$ for $3 \leq \ell \leq k$ this gives us $f(\mH_n) \to \infty$, implying Condition~\emph{(Mii)} for $n$ large enough. As $f(\mH_n) \to \infty$ we also know that $1 - 1/f(\mH_n) \to 1$. Now by definition of $f$ we have $v(\mH_n) / f(\mH_n) \geq v(\mH_n) / d(\mH_n)^{1/(k-1)}$ which by \emph{(M3)} goes to infinity. This gives us Condition~\emph{(Miii)} for $n$ large enough and the desired result follows.
\end{proof}

The following notion plays a crucial role in the proof of Theorem~\ref{thm:MakerWinCriterion} and is a natural generalization of the notion that a set is $(\delta,k)$--\emph{Szem\'eredi} as defined by Conlon and Gowers~\cite{CG10}.
\begin{definition}[$\delta$--stable]
Let $\mF$ be a hypergraph and  $0<\delta<1$. We say that a subset $T
\subseteq V(\mF)$ of the vertices is $\delta$\emph{-stable} if every subset of $S \subseteq T$ of size $|S| \geq \delta|T|$ contains an edge of $\mF$.
\end{definition}
%
	Equivalently, $T$ is called $\delta$--stable, if the subhypergraph of $\mF$ induced by $T$ has independence number less than $\delta \, |T|$.
%

Maker's strategy will consist of \emph{picking} (but not necessarily occupying) elements uniformly at random from among all elements he has not previously picked. With this rule it is guaranteed that the set of elements Maker picks is uniformly random. Then, if possible, Maker occupies that picked element in the game, otherwise he occupies an arbitrary free element. We will prove that with positive probability Maker wins using this strategy by showing that a $\delta$--fraction of the elements Maker \emph{picked} he was also able to \emph{occupy}. If we ensure that the set of vertices occupied by Maker is  $\delta$--stable, it then follows that Maker's set of vertices contain an edge with positive probability.

\medskip

Given some finite set $S$ and $0 < p < 1$ we will use the notation $S_p$ to refer to the \emph{binomial random set} that is obtained by picking each element of $S$ independently with probability $p$. For any given subset $T \subset S$ we therefore have $\mathbb{P}(S_p = T) = p^{|T|} (1-p)^{|S| - |T|}$. On the other hand, given $0 \leq M \leq n$ the \emph{uniform random set} $S_M$ is obtained by assigning each subset $T \subset S$  of size $|T| = M$ the same probability $\mathbb{P}(S_M = T) = 1/ \binom{n}{M}$.

The key ingredient to prove the existence of a winning strategy for Maker is the following statement that says that for $\mH$ as in Theorem~\ref{thm:MakerWinCriterion}, $V(\mH)_M$ is $\delta$--stable for suitable $M$ and $\delta$.
\begin{theorem}\label{thm:decay}
	For every $k \geq 2 $ and for every constant $c_1 \geq k$
        there exist  constants $\delta = \delta (k,c_1) < 1$  and $\tilde{c} = \tilde{c}(k,c_1) >0$ such that the following holds. If $\mH$ is a $k$--uniform hypergraph satisfying
	\begin{align*}
		 \text{(Mi)} \enspace \Delta_1(\mH)\leq c_1 \,
          \density{\mH}, \qquad  \text{(Mii)} \enspace f(\mH) > 1,
          \qquad \text{(Miii)} \enspace \frac{v(\mH)}{f(\mH)} \left(1 - \frac{1}{f(\mH)} \right) \geq \tilde{c}
	\end{align*}
	then
	\begin{equation*}
		\PP{V(\mH)_{M} \text{ is not } \delta \text{-stable}} < 3 \exp\left(-M \, \frac{1}{c_1 \; 2^{k+2}} \right) ,
	\end{equation*}
	for every $M  \geq  2 \, \floor{ v(\mH) / f(\mH)}$.	
\end{theorem}
We start by showing how to deduce Theorem~\ref{thm:MakerWinCriterion} assuming the statement of Theorem~\ref{thm:decay}. The proof of Theorem~\ref{thm:decay} will be given immediately afterwards.

\subsection{Proof of Theorem~\ref{thm:MakerWinCriterion} from Theorem~\ref{thm:decay}}

Fix an arbitrary strategy $S_B$ for Breaker. Maker will play according to the following random strategy: in each round he \emph{picks} an element uniformly at random from among all elements of $V(\mH)$ that he has not previously picked. If this element was not already occupied by either Maker or Breaker, then Maker occupies it. Otherwise he occupies an arbitrary free vertex and ‘forgets’ about it for the rest of the game, i.e. he doesn't consider it as picked and can potentially pick it at a later point. Note the subtle difference between picking and occupying a vertex: occupying is the act of actually choosing a vertex in the process of the game and picking can, depending on whether the vertex was already occupied by either player, be merely in the mind of Maker. We label an element picked by Maker as a \emph{failure}, if that element was already occupied by Breaker. We will show that this random strategy succeeds with positive probability against $S_B$, implying that $S_B$ is not a winning strategy. Since $S_B$ was arbitrary, this shows that Maker must have a winning strategy.

\noindent Let $\delta = \delta(k,c_1)<1$ be chosen according to Theorem~\ref{thm:decay} and define $c = (1-\delta)/4 >0$ to be a  constant. Let $q \leq c f(\mH) -1 $ and consider the first
\begin{equation} \label{eq:limitrounds}
	M = 2 \; \bigfloor{\frac{ v(\mH)}{ f(\mH)}} \leq \frac{1-\delta}{2} \;
 \frac{v(\mH)}{q+1}
\end{equation}
rounds of the game. We may consider the set of elements that Maker picked in these $M$ rounds as the uniform random set $V(\mH)_{M}$. Note that some of his elements may be failures. We will now upper bound the probability that Maker's $i$--th move, which we refer to as $m_i$, was a failure. Clearly this probability is upper bounded by the probability that his $M$--th move is a failure since in every round the number of potential failures does not decrease and the number of vertices  Maker picks from strictly decreases. Note that in the first $M-1$ rounds, Maker picked exactly $M-1$ vertices. So, in round $M$, there are $v(\mH)-M+1$ available vertices to pick from. The potential failures are among the vertices occupied by Breaker and hence their number is at most $q \, (M-1)$.  Using~\eqref{eq:limitrounds} it follows for every $i \in [M]$ that
\begin{equation*}
	\PP{m_i \text{ a failure}} \leq \PP{m_M \text{ a failure}} \leq \frac{q \, (M-1)}{v(\mH)-(M-1)} \leq  \frac{q \; M}{v(\mH)-M} \leq \frac{1-\delta}{2}.
\end{equation*}
The probability that Maker has more than $(1-\delta) \, M$ failures is now at most the probability that among $M$ independent Bernoulli trials with failure probability $(1-\delta)/2$ there exist  more than $(1-\delta) \, M$ failures, which is less than $1/2$ by Markov's inequality. In other words, with probability at least $1/2$, at least $\delta \, M$ elements picked by Maker are not failures.

By Theorem~\ref{thm:decay} we have $\PP{V(\mH)_{M} \text{ is not } \delta \text{-stable}} < 3 \exp\left(-M(c_1 \; 2^{k+2})^{-1} \right)$. Now setting $\bar{c}$ to be the maximum of $ c_1 2^{k+1}(\log(3)+ \log(4))+1$ and the according value of $\tilde{c}$ in Theorem~\ref{thm:decay}, and using that $f(\mH) \leq \bar{c}^{-1}v(\mH)$  by Condition~(Miii), one can verify that the probability that the uniform random set $V(\mH)_M$ is not $\delta$--stable is at most $1/4$. Consequently, with probability at least $1/4$, the vertices occupied by Maker (of which there are at least $\delta M$) contain an edge of $\mH$. \hfill $\square$


\subsection{Proof of Theorem~\ref{thm:decay}}

The heart of the proof of Theorem~\ref{thm:decay} is the following statement due to Janson, {\L}uczak and Ruci\'nski \cite{JTR90-paper}. We use here the version stated in Theorem $2.18$, \emph{(ii)} in~\cite{JLR00}.
\begin{theorem}[Janson-{\L}uczak-Ruci\'nski~\cite{JTR90-paper}]\label{thm:JLR}
	Let $S$ be a set and let $0 < p < 1$.  For a set $T \subseteq S$ denote by $\Ind{T}$  the indicator random variable for the event that $T \subseteq S_p$. Let $\mS \subset \mP(S)$ be a family of subsets of $S$ and let $X = \sum\limits_{T \in \mS} \Ind{T}$. Then
	\begin{equation*}
		\PP{X = 0} \leq \exp\Bigg(-\frac{\EE{X}^2}{\underset{\substack{(T,T') \in \mS^2 \\ T \cap T’ \neq \emptyset}}{ \sum} \!\!\! \EE{\Ind{T} \Ind{T'}}}\Bigg).
	\end{equation*}
\end{theorem}

We want to apply Theorem~\ref{thm:JLR} with $S = V(\mH)$, $\mS = E(\mH)$ and $p  = 1/f(\mH)$. Note that $p < 1$ due to Condition~(Mii). For $X = \sum_{e \in \mH} \Ind{e}$ we have, by linearity of expectation, that $\EE{X} = e(\mH) \; p^k$. Note that $\EE{\Ind{e}\Ind{e'}} = p^{2k-|e \cap e'|}$ for $e,e' \in \mH$ and therefore
\begin{align*}
	\underset{\substack{(e,e') \in \mH^2 \\ e \cap e' \neq \emptyset}}{\sum} \; \EE{\Ind{e} \Ind{e'}} & =  \sum_{e \in \mH} \; \sum_{\emptyset \neq T \subseteq e} \sum_{\substack{e' \in \mH \\ e \cap e' = T}} p^{2k-|T|} \leq \sum_{e \in \mH} \; \sum_{\emptyset \neq T \subseteq e} \! \deg(T) \, p^{2k-|T|}	\\
	& \leq e(\mH) \, \left( 2^k - 1 \right) \max_{\emptyset \neq T \subseteq V(\mH)} \left(\deg(T) \, p^{2k-|T|}\right) \\
	& \leq 2^{k} \, e(\mH) \; \max_{1 \leq \ell \leq k}\left( \Delta_{\ell}(\mH) \, p^{2k-\ell} \right) \\
	& = 2^{k} \, e(\mH) \, p^{2k-1} \; \max\left( \max_{2 \leq \ell \leq k} \left( \frac{\Delta_{\ell}(\mH)}{p^{\ell-1}} \right), \, \Delta_1(\mH) \right) \\
	& = 2^k \, \frac{\mathbb{E}(X)^2}{p \; v(\mH) \, \density{\mH}} \; \max\left( \max_{2 \leq \ell \leq k} \left( \frac{\Delta_{\ell}(\mH)}{p^{\ell-1}} \right), \, \Delta_1(\mH) \right).
\end{align*}
Using Condition~\emph{(Mi)} we now get
\begin{align*}
	\underset{\substack{(e,e') \in \mH^2 \\ e \cap e' \neq \emptyset}}{\sum} \; \EE{\Ind{e} \Ind{e'}} & \leq 2^{k} \, \frac{\EE{X}^2}{p \, v(\mH)} \; \max\left(\max_{2 \leq \ell \leq k} \left( \frac{\Delta_{\ell}(\mH)}{\density{\mH} \, p^{\ell-1}} \right), \, c_1 \right) \\
	& \leq 2^{k} \, \frac{\EE{X}^2}{p \, v(\mH)} \; \max\left(\max_{2 \leq \ell \leq k} \left( \frac{1}{\left( f(\mH) \; p \right)^{\ell-1}} \right), \, c_1\right) = c_1 \, 2^{k} \, \frac{\EE{X}^2}{p \, v(\mH)} \,.
\end{align*}
where the last inequality follows from the definition of $f(\mH)$ and the last equality follows from the fact that $f(\mH) \, p =1 $ and $c_1 \geq k$. Now, using the estimate from Theorem~\ref{thm:JLR}, we get
\begin{equation} \label{eq:noedges_binomial}
	\PP{V(\mH)_p \text{ contains no edge of } \mH} = \PP{X = 0} \leq \exp (-c' \, v(\mH) \, p)
\end{equation}
where $c' = 1/(c_1 \, 2^k)$. The following lemma will now allow us to bound the probability of a uniform random set fulfilling our desired property by the probability that a corresponding binomial random set fulfils it. A proof can be found in the Appendix.
\begin{lemma}\label{lem:randomsubset}
	Let $X \sim \mB(n,p)$ and let $\mP$ be a monotone decreasing family of subsets of $[n]$.Then there exists a constant $C>0$ such that if $\sqrt{np(1-p)} > C $, then  $\PP{[n]_{\floor{np}} \in \mP} \leq 3 \, \PP{[n]_p \in \mP}$.
\end{lemma}
Since the property of ‘not containing an edge of $\mH$’ is monotone decreasing, we can choose $\tilde{c}$ such that we can apply Lemma~\ref{lem:randomsubset} to restate~\eqref{eq:noedges_binomial} for the uniform random set model as follows:
\begin{equation}\label{eq:noedges}
	\PP{V(\mH)_{\bar{M}} \text{ contains no edge of } \mH} \leq 3 \, \PP{V(\mH)_p \text{ contains no edge of } \mH ) \leq 3 \, \exp(-c'\bar{M}}
\end{equation}
for any  $\bar{M}  \geq \floor{v(\mH)/f(\mH)}$.

We are now ready to finish the proof. Let $M  \geq 2 \, \floor{v(\mH)/f(\mH)}$ and let $\delta = \delta(k,c_1) \in (1/2,1)$ be such that $(1-\delta)(1-\ln(1-\delta)) < c'/4$. To see that this is indeed possible, note that for $x \in (0,1)$ the function $g(x) = (1-x)(1-\ln(1-x)) $ satisfies $g(x) \rightarrow 0$ as $x \rightarrow 1$. Consider pairs $(T,T')$ where $T \subset V(\mH)$ with $|T| = M$ and $T' \subseteq T$ is such that $|T'| = \delta \, M$ and $T'$ does not contain an edge of $\mH$. Using Inequality~\eqref{eq:noedges} with $\delta \, M > \floor{v(\mH)/f(\mH)}$, we can estimate the number of choices for a set $T'$ of size $\delta \, M$ that contains no edge of $\mH$ by
\begin{equation*}
	3 \, \exp(-c' \delta M) \, \binom{v(\mH)}{\delta \, M} \leq 3 \, \exp\Big(\! -c' \frac{M}{2} \Big) \, \binom{v(\mH)}{\delta \, M}.
\end{equation*}
Hence, we can upper bound the number of pairs $(T,T')$ as described above by
\begin{equation*}
	3 \, \exp\Big(\! -c' \frac{M}{2}\Big) \, \binom{v(\mH)}{\delta \, M} \, \binom{v(\mH) - \delta M}{(1-\delta) \, M} = 3 \, \exp\Big(\! -c' \frac{M}{2} \Big) \, \binom{M}{(1-\delta) \, M} \, \binom{v(\mH)}{M}.
\end{equation*}
We can therefore upper bound the number of choices for a set $T$ of size $M$ containing a subset of size $\delta M$ that does not contain an edge of $\mH$ by
\begin{equation*}
	3 \, \exp\Big(\! -c' \frac{M}{2} \Big) \, \binom{M}{(1-\delta) \, M} \, \binom{v(\mH)}{M} \leq 3 \, \exp \left( M (-c'/2 + (1-\delta)(1-\ln(1-\delta))) \right) \, \binom{v(\mH)}{M}.
\end{equation*}
Hence we get
\begin{equation*}
	\PP{V(\mH)_{M} \text{ is  not } \delta \text{-stable }} \leq  3 \, \exp \left( M(-c'/2 + (1-\delta)(1-\ln(1-\delta))) \right) \leq 3\exp\left( -M \frac{c'}{4}\right)
\end{equation*}
where the last inequality follows by choice of $\delta = \delta(k,c_1)$. \hfill $\square$

\section{Proof of Theorem~\ref{cor:BreakerWinCriterion} -- Breaker's Strategy} \label{sec:BreakerProof}

We will derive Theorem~\ref{cor:BreakerWinCriterion} from the following stronger combinatorial statement.
\begin{theorem}[Breaker Win Criterion] \label{thm:StrongerBreakerWinCriterion}
For every $k \geq 2$ and $t > (2k)^k$ the following holds. If $\mH$ is a $k$--uniform hypergraph, then Breaker has a winning strategy in $\bG(\mH;q)$ provided that
\begin{equation*}
	q > 4 \max \left( \left( (2 \, v(\mH))^{1/t} \; \Delta_1(\mH) \; k e \right)^{\frac{1}{k-1}},\; 2 k^2 t^3 \left(\max\limits_{2 \leq \ell \leq k-1} \left( \Delta _{\ell}(\mH) \left( (tk)^{tk} \; k^t \; v(\mH)^2 \right)^{\frac{k}{t^{1/k}}} \right)^{\frac{1}{k-\ell}} +2 \right) \right).
\end{equation*}
\end{theorem}

Note that $e$ above denotes the base of the natural logarithm and should not be confused with the number of edges. We start by giving a proof of Theorem~\ref{cor:BreakerWinCriterion} using Theorem~\ref{thm:StrongerBreakerWinCriterion}. Then we define the necessary concepts for the remainder of the section. Following this we present the two main strategies for Breaker and prove their correctness. Finally we prove Theorem~\ref{thm:StrongerBreakerWinCriterion} using these ingredients.

\begin{proof}[Proof of Theorem~\ref{cor:BreakerWinCriterion} from Theorem~\ref{thm:StrongerBreakerWinCriterion}]

Let $k \geq 2$ and $\epsilon > 0$ be given and set $t = \log v(\mH)$. Assume that $v(\mH)$ is large enough such that $\log v(\mH) > (2k)^k$. Using $e = v(\mH)^{1 / \log v(\mH)}$ it is straightforward to check that
\begin{equation*}
	\left( (2n)^{1/t}\Delta_1(\mH) \; ke \right)^{\frac{1}{k-1}} \leq C_1' \, \Delta_1(\mH)^{\frac{1}{k-1}}
\end{equation*}
for some constant $C_1' = C_1'(k) > 0$. Similarly for $v(\mH)$ sufficiently large we can upper bound the term
\begin{equation*}
	2 k^2 t^3 \left(\max_{2 \leq \ell \leq k-1} \left( 2k \; \Delta _{\ell}(\mH)  \left( v(\mH)^2 \; (tk)^{tk}  \right)^{\frac{k}{t^{1/k}}}\right)^{\frac{1}{k-\ell}} +2 \right) \leq C_2' \, v(\mH)^{C_3' \frac{\log\log v(\mH)}{\log^{1/k} v(\mH)} } \max_{2\leq \ell \leq k-1} \left(\Delta_{\ell}(\mH) \right)^{\frac{1}{k-\ell}}
\end{equation*}
for some constants $C_2' = C_2'(k) > 0$ and $C_3' = C_3'(k) > 0$. Note that $\log\log v(\mH) / \log^{1/k} v(\mH) \rightarrow 0$ and so for $v(\mH)$ large enough this will be at most $v(\mH)^{\epsilon} \, \max_{2\leq\ell \leq  k-1} \left( \Delta_{\ell}(\mH)^{\frac{1}{k-\ell}}\right)$. Choose $C_1 = C_1(k) \geq \max(C_1', C_2', 4)$ and $v_0 = v_0(k)$ large enough, giving us the statement that Breaker has a winning strategy if
	\begin{equation} \label{eq:BreakerBiasBound}
		q \geq C_1 \max \left(\Delta_1(\mH)^{\frac{1}{k-1}}, \; \max_{2 \leq \ell \leq k-1} \left( \Delta_{\ell}(\mH)^{\frac{1}{k-\ell}}\right) \, v(\mH)^{\epsilon} \right).
	\end{equation}
	From this, the desired statement of Theorem~\ref{cor:BreakerWinCriterion} immediately follows.
\end{proof}

\subsection{Preliminaries for the proof of Theorem~\ref{thm:StrongerBreakerWinCriterion}}

One of the most important results in the area of positional games is the Erd\H{o}s-Selfridge Theorem~\cite{ES73}, the biased version of which is due to Beck~\cite{Be85}. It ensures that Breaker can do at least as well as the expected outcome when both players act randomly. We will use the following consequence of it heavily in the proof of Theorem~\ref{thm:StrongerBreakerWinCriterion}.

\begin{theorem}[Biased Erd\H{o}s-Selfridge Theorem~\cite{Be85}] \label{thm:BiasedErdosSelfridge}
	For every hypergraph $\mH$ and integer $q \geq 1$ the following holds. If Breaker plays as the second player, he can keep Maker from covering more than
	\begin{align} \label{eq:BiasedErdosSelfridge}
	(q+1) \; \sum_{H \in \mH} \left( \frac{1}{q+1} \right)^{|H|}
	\end{align}
	winning sets in $\bG(\mH;q)$. If he plays as the first player, then one can omit the first $(q+1)$ factor.
\end{theorem}

\noindent We will also need the following simple yet powerful remark.
\begin{remark} \label{rmk:biasdoubling}
	If Breaker has a winning strategy for some positional game $\bG(\mH;q)$ where he is allowed to make \emph{at most} $q$ moves each round, then he also wins if he has to make \emph{exactly} $q$ moves each round. It follows that if he has a winning strategy for some game $\bG(\mH_1;q_1)$ and a winning strategy for another game $\bG(\mH_2;q_2)$, then he can combine these two strategies to define a winning strategy in $\bG(\mH_1 \cup \mH_2;q_1+q_2)$.
\end{remark}
\noindent This remark will be used extensively throughout the proof. Furthermore, we will need the following definitions, which are based on those developed in~\cite{BL00}.
\begin{definition}[Set-Theoretic definitions]
	Given some hypergraph $\mH$, we define the following:
	\begin{enumerate} \setlength\itemsep{-0em}
		\item[--]	a \emph{$t$--cluster} is any family of distinct edges $\{ H_1, \dots, H_t \} \subset \mH$ satisfying $|\bigcap_{i = 1}^{t} H_i| \geq 2$,
		\item[--]	an \emph{almost complete solution $(H^{\circ},h)$} is a tuple consisting of a set  $H^{\circ} \subseteq V(\mH)$ as well as an element $h \notin H^{\circ}$ so that $H = H^{\circ} \cup \{ h \}$ is an edge in $\mH$,
		\item[--]	a \emph{$t$--fan} is a family of distinct almost complete solutions $\{ (H_1^{\circ},h_1), \dots , (H_t^{\circ},h_t) \}$ in $\mH$ satisfying $|\bigcap_{i = 1}^{t} H_i^{\circ}| \geq 1$ and it is called \emph{simple} if $|H_i^{\circ} \cap H_j^{\circ} | = 1$ for all $1 \leq i < j \leq t$,
		\item[--] 	a \emph{$t$--flower} is a $t$--fan satisfying $|\bigcap_{i = 1}^{t} H_i^{\circ}| \geq 2$.
	\end{enumerate}
	For each $t$--fan in $\mH$ we call the $h_i$ the \emph{open elements}, the $H_i^{\circ}$ the \emph{major parts} and the elements of the intersection $\bigcap_{i = 1}^{t} H_i^{\circ}$ the \emph{common elements}.
\end{definition}
\begin{definition}[Game-Theoretic Definitions]
	At any given point in a positional game on a given hypergraph $\mH$, we call an almost complete solution $(H^{\circ},h)$ \emph{dangerous} if all elements of $H^{\circ}$ have been picked by Maker and $h$ has not yet been picked by either player. A fan or flower is dangerous if their respective almost complete solutions are.
\end{definition}
Observe that for a dangerous $t$--fan or $t$--flower we must have $h_i \notin H_j^{\circ}$ for all $1 \leq i,j \leq t$. In the following we will always assume that Breaker plays as second player. We say that a player \emph{occupies} a given $t$--fan or $t$--flower $(H_1^{\circ},h_1), \dots, (H_t^{\circ},h_t)$ if his selection of vertices contains $\bigcup_{i=1}^t H_i^{\circ}$. Similarly a player \emph{occupies} a $t$--cluster $H_1, \dots, H_t$ if his selection of vertices contains $\bigcup_{i=1}^t H_i$.

\subsection{Two important strategies for Breaker}

The following two lemmata give us strategies that we will use to construct a larger strategy in the proof of Theorem~\ref{thm:StrongerBreakerWinCriterion}. Note that in the statement of the lemma we do not care about which player covers the open elements of a fan.
\begin{lemma}\label{lem:fans}
	For every integer $k\geq 2$ and $t\geq 1$ the following holds. If $\mH$ is a $k$--uniform hypergraph, then Breaker with a bias of $q > \left( (2 \, v(\mH))^{1/t}\; \Delta_1(\mH) \; ke \right)^{1/(k-1)} $ has a strategy that prevents Maker from occupying $1/2 \, {q \choose t}$ simple $t$--fans in the game $\bG(\mH;q)$.
\end{lemma}
\begin{proof}
	Let $\mF = \big\{ \bigcup_{i=1}^{t} H_i^{\circ} \mid \{ (H_1^{\circ},h_1), \dots , (H_t^{\circ},h_t) \} \; \text{simple } t\text{-fan in } \mH \big\}$ be the hypergraph of all simple $t$--fans in $\mH$. We want to apply Theorem~\ref{thm:BiasedErdosSelfridge}, so we estimate
	\begin{equation*}
		(q+1) \; \sum_{F \in \mF} \left( \frac{1}{q+1} \right)^{|F|} \leq (q+1) \left( v(\mH) \; \frac{\Delta_1(\mH)^t \; (k-1)^t }{t!} \right) \left( \frac{1}{q+1} \right)^{t(k-2) + 1}.
	\end{equation*}
	This inequality holds because there are $v(\mH)$ ways to fix the common element of a simple $t$--fan, $\Delta_1(\mH)^t$ is an upper bound on the number of $t$--tuples of edges containing the fixed common element and there are $(k-1)^t$ ways of fixing the corresponding open elements. Note that an open element is never a common element by definition. Furthermore, $t!$ takes care of the symmetry and each simple $t$--fan is of size $t(k-2) + 1$. We therefore get
	\begin{align*}	
		(q+1) \, \sum_{F \in \mF} \left( \frac{1}{q+1} \right)^{|F|} & \leq v(\mH) \left( \frac{\Delta_1(\mH) \; ke}{t \; q^{k-2}} \right)^t = v(\mH) \left( \frac{\Delta_1(\mH) \; ke}{q^{k-1}} \right)^t \left( \frac{q}{t} \right)^t \\
		& < v(\mH) \; \frac{1}{2 \, v(\mH)} \left(\frac{q}{t} \right)^t \leq \frac{1}{2} \binom{q}{t}.
	\end{align*}
	The claim now follows by applying Theorem~\ref{thm:BiasedErdosSelfridge}.
\end{proof}
\begin{lemma} \label{lemma:dangerousflowers}
	For every integer $k \geq 2$ and $t > (2k)^k$ the following holds. If $\mH$ is a $k$--uniform hypergraph, then Breaker with a bias of
	\begin{equation} \label{eq:dangerousflowers}
		q > \max\limits_{2 \leq \ell \leq k-1} \left( \Delta_{\ell}(\mH) \; ((tk)^{tk} \; k^t \; v(\mH)^2)^{\frac{k}{t^{1/k}}} \right)^{\frac{1}{k-\ell}}
	\end{equation}
	has a strategy that prevents dangerous $t(q+1)$--flowers in $\bG(\mH;q)$
\end{lemma}
\begin{proof}
	Let $\mF = \big\{ \bigcup_{i=1}^{t} H_i \mid \{ H_1, \dots , H_t \} \; t\text{-cluster in } \mH \big\}$ be the hypergraph of all $t$--clusters in $\mH$. First we will show that Breaker can prevent $t$--clusters. Given some $t$--cluster $H_1, \dots, H_t$ let $\ell_i = \big| H_i \, \cap \, \bigcup_{j=1}^{i-1} H_j \big|$ for all $2 \leq i \leq t$. We call $(2,\ell_2, \dots, \ell_t)$ its \emph{intersection characteristic} and observe that $2 \leq \ell_i \leq k$ for $2 \leq i \leq t$. We will set $\ell_1 = 2$ for notational convenience. For any $\bell = (\ell_1, \dots, \ell_t) \in \{2\} \times [2,k]^{t-1}$ let $\mF (\bell)$ denote the set of edges in $\mF$ which come from some $t$--cluster with the intersection characteristic $\bell$ and observe that it is $v(\bell)$--uniform where
	\begin{equation} \label{eq:vbell}
		v(\bell) = 2 + \sum_{i = 1}^t (k - \ell_i) = k + \sum_{i = 2}^t (k - \ell_i).
	\end{equation}
	This follows since given any cluster $H_1, \dots, H_t$ with intersection characteristic $\bell$ we have $|\bigcup_{i=1}^{t} H_i| = v(\bell)$. There is the trivial upper bound $v(\bell) \leq tk$ for all $\bell \in \{2\} \times [2,k]^{t-1}$. Let $L = \{ \bell : \mF(\bell) \neq \emptyset \} \subseteq \{2\} \times [2,k]^{t-1}$ be the set of all intersection characteristics that actually occur in $\mH$. Now for any $\bell \in L$ we trivially have $t \leq \binom{v(\bell) - 2}{k - 2}$ which we restate as the lower bound
	\begin{equation} \label{eq:vbelllowerbound}
		v(\bell) \geq t^{1/k} \text{ for all } \bell \in L.
	\end{equation}

	Now for $\bell = (\ell_1, \dots, \ell_t) \in L$ observe that
	\begin{align*}
		|\mF(\bell)| & \leq \binom{v(\mH)}{2} \; \Delta_2(\mH) \; \prod_{i=2}^{t} \binom{k + \sum_{j=2}^{i-1} (k-\ell_i) - 2}{\ell_{i}-2} \; \Delta_{\ell_{i}}(\mH) \\
		& \leq \binom{v(\mH)}{2} \; \binom{v(\bell)-2}{k-2}^{t-1} \; \Delta_2(\mH) \; \prod_{i=2}^t \Delta_{\ell_{i}}(\mH)\leq v(\mH)^2 \; (tk)^{tk} \; \prod_{i=1}^t \Delta_{\ell_{i}}(\mH).
	\end{align*}	
	Here, the first inequality is justified by observing that there are $\binom{v(\mH)}{2}$ ways to fix two common elements and at most $\Delta_2(\mH)$ ways to choose the first edge $H_1$ of a $t$--cluster. The product counts ways to add the $i$--th additional edge $H_i$ for $2 \leq i \leq t$ by first fixing the intersection with the already established parts $\bigcup_{j = 1}^{i-1} H_j$ and then adding one of the at most $\Delta_{\ell_i}$ possible ways of picking $H_i$. The second inequality follows since by assumption $t > (2k)^k$ so that~\eqref{eq:vbelllowerbound} gives us $v(\bell) \geq 2k$ from which it follows that for all $2 \leq i \leq t$ we have
	\begin{equation*}
		\binom{k + \sum_{j=2}^{i-1} (k-\ell_i) - 2}{\ell_{i}-2} \leq {v(\bell)-2 \choose k-2}.
	\end{equation*}

	We now want to apply Theorem~\ref{thm:BiasedErdosSelfridge}, so we estimate
	\begin{align*}
		(q + 1) \sum_{F \in \mF} \left( \frac{1}{q+1} \right)^{|F|} & \leq (q + 1) \sum_{\ell_2 \in [2,k]} \!\! \cdots \!\! \sum_{\ell_t \in [2,k]} |\mF(\bell)| \; \left( \frac{1}{q + 1} \right)^{v(\bell)} \\
		& \leq (tk)^{tk} \; v(\mH)^2 \; (q + 1) \; \sum_{\bell \in L} \; \prod_{i=1}^t \Delta_{\ell_{i}}(\mH) \left( \frac{1}{q + 1} \right)^{v(\bell)}.
	\end{align*}
	where we have just inserted the previously stated upper bound on $|\mF(\bell)|$. We now split up the factor $\left(1/(q+1) \right)^{v(\bell)}$ using~\eqref{eq:vbell} to obtain
	\begin{align*}
		(q + 1) \sum_{F \in \mF} \left( \frac{1}{q+1} \right)^{|F|} & \leq(tk)^{tk} \; v(\mH)^2 \; \frac{1}{q+1} \; \sum_{\bell \in L} \; \prod_{i=1}^t \left( \Delta_{\ell_{i}}(\mH) \left( \frac{1}{q+1} \right)^{k - \ell_i} \right).
	\end{align*}
	Note that  we have $\Delta_{\ell}(\mH) \left( 1/(q+1) \right)^{k - \ell} =1$ for $\ell =k$ and  $\Delta_{\ell}(\mH) \left( 1/(q+1) \right)^{k - \ell} < 1$ for $2 \leq \ell < k$   due to the lower bound on $q$. Furthermore, since $\bell \in L$ is the intersection characteristic of a $t$--cluster in $\mH$, the number of indices  $1 \leq i \leq t$ for which $\ell_i < k$  must be at least $\bigceil{v(\bell)/k} \geq \bigceil{t^{1/k}/k}$. Now, due to~\eqref{eq:dangerousflowers} it follows that
	\begin{align*}	
		(q + 1) \sum_{F \in \mF} \left( \frac{1}{q+1} \right)^{|F|} & \leq (tk)^{tk} \; v(\mH)^2 \; k^t \; \left( \max_{2 \leq \ell \leq k-1} \Delta_{\ell}(\mH) \left( \frac{1}{q} \right)^{k - \ell} \right)^{\frac{t^{1/k}}{k}} < 1.
	\end{align*}	
	It follows, by applying Theorem~\ref{thm:BiasedErdosSelfridge}, that using a bias of  $q$, Breaker has a strategy to keep Maker from fully covering any $t$--cluster. Following this strategy, it is easy to see that Breaker will also keep Maker from creating a dangerous $t(q+ 1)$--flower at any point in the game. To see this, suppose that this is not the case and that Maker succeeds in creating such a dangerous flower. By repeatedly claiming the open element of this dangerous flower which has not yet been claimed and is the open element of the most almost complete solutions in the flower, Maker would be able to cover a $t$--cluster, as $t(q+ 1)/(q + 1) = t$, which is a contradiction.
\end{proof}

\subsection{Proof of Theorem~\ref{thm:StrongerBreakerWinCriterion}}

In order to join the previous two strategies together, we will need the following simple auxiliary statement. We include its simple proof for the convenience of the reader.
\begin{lemma} \label{lemma:auxiliary}
	For every $q \geq 2$ and $t \geq2 $ the following holds: If $F$ is a graph on $q$ vertices with $ e(F) < q^2/2t^2$ then $F$ has at least $1/2 \, {q \choose t}$ independent sets of size $t$.
\end{lemma}
\begin{proof}
	The number of subsets of $V(F)$ of size $t$ that are not independent is upper bounded by \[ e(F) {q-2\choose t-2} \leq e(F) \left( \frac{t^2}{q^2}\right){q\choose t} < \frac{1}{2} {q \choose t} \] since $e(F) < q^2/2t^2$.
\end{proof}
\noindent We are now ready to prove Theorem~\ref{thm:StrongerBreakerWinCriterion}. Let $k \geq 2$ and $t > (2k)^k$ be given and let
\begin{equation*}
	q > 4 \max \left( \left( (2 \, v(\mH))^{1/t}\Delta_1(\mH) \; ke \right)^{\frac{1}{k-1}}\; ,\; 2 k^2 t^3 \left(\max\limits_{2\leq \ell \leq k-1} \left( \Delta _{\ell}(\mH) ((tk)^{tk} \; k^t \; v(\mH)^2)^{\frac{k}{t^{1/k}}} \right)^{\frac{1}{k-\ell}} +2 \right) \right).
\end{equation*}
Breaker will play according to the following three strategies, splitting his bias as $q = q/2 + q/4 + q/4$. Note that in case Breaker does not need all his moves to play according to one of the strategies, he plays them arbitrarily, which cannot hurt him.
\begin{itemize}
	\item[\emph{SB1:}]Using $q/4$ moves, he will play according to Lemma~\ref{lem:fans} and thus preventing Maker from occupying $1/2 \, {q/4  \choose t}$ simple $t$--fans.
	\item[\emph{SB2:}]Using $\overline{q} =\max_{2\leq \ell \leq k-1} \Big( \Delta_{\ell}(\mH) \left( (tk)^{tk} k^t \, v(\mH)^2 \right)^{k/t^{1/k}} \Big)^{1/(k-\ell)}  + 1 < q/4$ moves, he will play according to Lemma~\ref{lemma:dangerousflowers} and hence preventing dangerous $t(\overline{q}+1)$--flowers from appearing.
	\item[\emph{SB3:}]Using $q/2$ moves, Breaker will occupy all open elements of any dangerous almost complete solution.
\end{itemize}

First of all note that Maker can play according to \emph{SB1} and \emph{SB2} since
\begin{equation*}
	q/4 > \Big( (2 \, v(\mH))^{1/t}\Delta_1(\mH) \; ke \Big)^{1/(k-1)} \quad \text{and} \quad \overline{q} > \max_{2\leq \ell \leq k-1} \Big( 2k \; \Delta _{\ell}(\mH) \, \left( v(\mH)^2 \; (tk)^{tk} \right)^{\frac{k}{t^{1/k}}}\Big)^{1/(k-\ell)}.
\end{equation*}

We can combine these strategies due to Remark~\ref{rmk:biasdoubling} and will now prove by induction, that after each of Breaker's moves there is no dangerous almost complete solution. Clearly this implies that Breaker's strategy is indeed a winning strategy. Initially there is obviously no dangerous almost complete solution. So suppose the result is true in round $r-1$. In round $r$ Maker claims an element ($w$ say). Then every new dangerous almost complete solution must contain $w$. Therefore they all belong to the same dangerous fan (with common element $w$). In order to complete the inductive step, we have to show that the size of this dangerous fan is not more than $q/2$ as  Breaker can then occupy all open elements in this dangerous fan (SB3), which completes the inductive step. Indeed, using a bias of $q/2$ Breaker has a strategy that avoids dangerous $q/2$--fans at any point in the game.

Suppose Maker succeeds in occupying a dangerous $(q/2)$--fan $(H_1^{\circ},h_1), \dots , (H_{q/2}^{\circ},h_{q/2})$. Construct an auxiliary graph $F$ whose vertices are the almost complete solutions of this fan and an edge between $(H_i^{\circ},h_i)$ and $(H_j^{\circ},h_j)$ indicates that $|H_i^{\circ} \cap H_j^{\circ}| \geq 2$ where $1 \leq i < j \leq q/2$. Recall that using $\overline{q}$ moves according to \emph{SB2}, Breaker prevents dangerous $t(\overline{q}+1)$--flowers from appearing. Therefore the maximum degree in $F$ is bounded by $\Delta(F) \leq (t(\overline{q}+1)-2){k-1 \choose 2} \leq t(\overline{q}+1)k^2$	 and hence $e(F) \leq \frac{1}{2} \frac{q}{2}t(\overline{q}+1)k^2 < \frac{1}{2} \frac{(q/2)^2}{t^2}$ by choice of $q$. Therefore, by Lemma~\ref{lemma:auxiliary}, $F$ has at least $\frac{1}{2} {q/4\choose t}$ independent sets of size $t$. But that means that Maker occupied $\frac{1}{2} {q/4 \choose t}$ simple $t$--fans contradicting SB1. This establishes the claim that Breaker has a strategy that avoids dangerous $q/2$--fans and finishes the proof. \hfill $\square$

\section{Proof of Theorem~\ref{thm:ThresholdGeneralizedVdWGames-Proper} -- Rado Games} \label{sec:vdWGames}

The goal of this section is to prove the statement in Theorem~\ref{thm:ThresholdGeneralizedVdWGames-Proper}, that is to show that the threshold bias of the Maker-Breaker $(A,\bb)$--game on $[n]$ satisfies $q(\mS_0(A,\bb,n)) = \Theta \left( n^{1/m_1(A)} \right)$ for a given positive and abundant matrix $A \in \ZZ^{r\times m}$ and vector $\bb \in \ZZ^r$. We start by establishing some preliminary results regarding linear systems of equations. We then obtain Maker's strategy through an application of Theorem~\ref{cor:MakerWinCriterion} and Breaker's strategy through an application of Theorem~\ref{cor:BreakerWinCriterion}. At the end of this section, we will also provide a proof for the easier statement regarding non-abundant matrices in Proposition~\ref{prop:ThresholdEasyGeneralizedVdWGames} and discuss the effect that solutions with repeated components of solutions have on the game.

\subsection{Preliminaries for Linear Systems}

We start with a couple of simple observation. If $A_1, A_2\in \ZZ^{r \times m}$ and $\bb_1, \bb_2 \in \ZZ^r$ are such that the solution sets are identical, that is $S(A_1,\bb_1) = S(A_2,\bb_2)$, then the corresponding game hypergraphs are of course also identical. Hence for any  invertible matrix $P\in \ZZ^{r\times r}$, the $(A,\bb)$--game on $[n]$ is the same as the $(P \! \cdot \! A,P \cdot \bb)$--game on $[n]$ for any $n\in \NN$. Furthermore, if $A$ was abundant then so is $P \! \cdot \! A$. In particular, applying elementary row operations (multiplying a row by a non-zero constant, adding a row to another row, switching two rows) to any matrix does not change the solution space or the underlying game, and does not make an abundant matrix non-abundant.

We will now continue by giving two basic bounds for the number of proper solutions. Given any matrix $A \in \ZZ^{r \times m}$ and vector $\bb \in \ZZ^r$, we remark that we have the upper bound
	\begin{equation} \label{eq:trivialupperbound}
		\big| \, S_0(A,\bb) \cap [n]^m \big| \leq \big| \, S(A,\bb) \cap [n]^m \big| \leq n^{m - \rank{A}}.
	\end{equation}
Indeed, taking a subset $Q\subseteq [m]$ of the column indices with $\rank{A} = |Q| = \rank{A^Q}$ and setting the $m - \rank{A}$  entries in ${\widebar{Q}}$ of a solution $\bx \in S(A, 0)$ arbitrarily, the entries in $Q$ are determined uniquely.
	
The following lemma, the proof of which is based on a construction by Janson and Ruci\'nski~\cite{JR11}, establishes that (\ref{eq:trivialupperbound}) is tight up to a constant factor.
\begin{lemma} \label{lemma:solutionslowerbound}
	For every positive and abundant matrix $A \in \ZZ^{r\times m}$ and vector $\bb \in \ZZ^r$ such that $S(A,\bb) \neq \emptyset$ there exist constants $c_0 = c_0(A,\bb) > 0$ and $n_0 = n_0(A,\bb) \in \NN$ such that for every $n \geq n_0$
	\begin{equation*}
		| \, S_0(A,\bb) \cap [n]^m| \geq c_0 \, n^{m-\mathrm{rk}(A)}.
	\end{equation*}
\end{lemma}
\begin{proof}
	We need to construct many {\em proper} positive solutions to $A \cdot \bold{x} = \bold{b}$. First we prove that there exists at least one to the homogenous system. Since $A$ is positive, we can take a positive solution $\bx^*  = (x_1^*,\dots,x_m^*)\in S(A,\bold{0}) \cap \NN^m$ that has the minimum number of pairs of equal entries. We claim that $\bx^*$ is proper.

	Assume to the contrary that there are two column indices (say $1$ and $2$), such that the corresponding entries of $\bx^*$ are equal, that is $x^*_1=x^*_2$. If there was a solution $\bold{y} \in S(A,\bold{0}) \cap \NN^m$ with $y_1 \neq y_2$, then consider a solution
	\begin{equation*}
		\bold{w} = \bx^* + \alpha \, \bold{y} \in S(A,\bold{0}) \cap \NN^m,	
	\end{equation*}
	where $\alpha \in \NN \setminus \{ (x_r^* - x_s^*)/(y_s-y_r) : 1\leq r, s \leq m, y_s\neq y_r \}$ is chosen arbitrarily. By the definition of $\alpha$, for all $r, s$ with $x_r^* \neq x_s^*$ we also have $w_r\neq w_s$, furthermore $w_1\neq w_2$. Consequently $\bold{w}$ has less pairs of equal entries than $\bx^*$, contradicting the choice of $\bx^*$.

	If $\bx^*$ is not proper, then for every positive solution $\by \in S(A,\bold{0}) \cap \NN^m$ we must have $y_1 = y_2$. This implies that the vector $(1,-1, 0, \ldots , 0)$ is orthogonal to {\em every} vector in $S(A,\bold{0})$, so it is a (rational) linear combination of the rows of $A$. Consequently some  row in $A$ can be replaced by a row $(c,-c, 0, \ldots, 0)$ where $c\in \ZZ\setminus \{0 \}$, using elementary row operations. Such a matrix is not abundant though, since the deletion of the first two columns reduces its rank, a contradiction.

	It follows that the positive solution $\bx^*  \in S_0(A,{\bold 0}) \cap \NN^m$ is proper. To construct many proper positive solutions to $A \cdot \bold{x} = \bold{b}$ in $[n]$, we choose a solution $\hat{\bx} = (\hat{x}_1, \dots, \hat{x}_m) \in S(A,\bb)$ as well as $m - \rank{A}$ linearly independent solutions $\bx_1, \dots \bx_{m-\rank{A}} \in S(A,{\bold 0})$. Let $s^*$ and $\hat{s}$ be the maximum absolute value of the entries of $\bx^*$ and $\hat{\bx}$, respectively, and $s$ the maximum absolute value of any entry in any of the vectors $\bx_1, \dots , \bx_{m-\rank{A}}$. Define $a(n) = \floor{n / (s^* + 1)}$ and set
	\begin{equation*}
		S(n) = \Bigg\{ \hat{\bx} + a(n) \, \bx^* + \sum_{i=1}^{m-\rank{A}} w_i \, \bx_i : w_i \in \ZZ, \, |w_i| < \frac{a(n)-2\hat{s}}{2s(m-\rank{A})} \Bigg\} \subseteq \ZZ^m.
	\end{equation*}
	Since $A\cdot \hat{\bx} = \bold{b}$ and $A\cdot \bx^* = A\cdot \bx_i = \bold{0}$, we have $S(n) \subseteq S(A,\bb)$. Let $\bx =(x_1, \dots, x_m) \in S(n)$ and observe that for $n$ large enough
	\begin{align*}
		x_i & > \hat{x}_i + a(n) \, {x}_{i}^* - (m-\rank{A}) \, s \, \frac{a(n) - 2\hat{s}}{2s(m-\rank{A})} \geq - \hat{s} + a(n) - \frac{a(n)}{2} +\hat{s} \geq 1
	\end{align*}
	as well as
	\begin{align*}
		x_i & < \hat{x}_i + a(n) \, x_{i}^* + (m-\rank{A}) \, s \, \frac{a(n) - 2\hat{s}}{2s(m-\rank{A})} \leq \hat{s} + n \, \frac{s^*}{s^* + 1} + n \, \frac{1}{2(s^* + 1)} - \hat{s} \leq n
	\end{align*}
	for every $i \in [m]$.  Consequently $S(n) \subseteq [n]^m$ for $n$ large enough.

	Now assume without loss of generality that $x_{1}^* < \dots < x_{m}^*$. It follows that
	\begin{align*}
		x_i & < \hat{x}_i + a(n) \, x_{i}^* + (m-\rank{A}) \, s \, \frac{a(n) - 2\hat{s}}{2s(m-\rank{A})} \\
		& = \left( \hat{x}_i - \hat{s} \right) + a(n) \left( x_{i}^* + \frac{1}{2} \right) \leq \left( \hat{x}_{i+1} + \hat{s} \right) + a(n) \left( x_{i+1}^* - \frac{1}{2} \right) \\
		& = \hat{x}_{i+1} + a(n) \, x_{i+1}^* - (m - \rank{A}) \, s \, \frac{a(n)-2\hat{s}}{2s(m-\rank{A})} < x_{i+1}
	\end{align*}
	for every $1 \leq i \leq m-1$, so $\bx$ is proper. Therefore $S(n) \subseteq S_0(A,\bb) \cap [n]^m$.

	Lastly observe that since $\bx_1, \dots \bx_{m-\rank{A}}$ are linearly independent, $S(n)$ contains
	\begin{equation*}
		\left( 2 \bigfloor{\frac{a(n) - 2\hat{s}}{2s(m-\rank{A})}}+1 \right)^{m-\rank{A}} \geq \left( \frac{1/(4 s^*+4)}{s \, \big( m-\rank{A} \big)} \, n \right)^{m-\rank{A}}
	\end{equation*}
	elements, where the lower bound holds for $n$ large enough. It follows that for
	\begin{equation*}
		c_0 = c_0(A,\bb) = \left( \frac{1/(4 \hat{s}_{\bold 0}+4)}{s \, \big( m-\rank{A} \big)} \right)^{m-\rank{A}} < 1
	\end{equation*}
	and $n_0 = \ceil{4 \, \hat{s} \, (\hat{s}_{\bold 0} + 1)}$ we have $| \, S_0 (A,\bb) \cap [n]^m| \geq c_0 \, n^{m-\rank{A}}$ for all $n \geq n_0$.
\end{proof}

\medskip

Recall the definitions from the introduction, especially $r_Q = r_Q(A) = \rank{A} - \rank{A^{\widebar{Q}}}$ as well as the definition of $m_1(A)$. We also introduce the additional notation that for any matrix $A \in \ZZ^{r \times m}$ and selection of row indices $R \subseteq [r]$, we let $A_R$ denote the matrix obtained by only keeping the rows indexed by $R$. Furthermore, a matrix $A$ is called {\em strictly balanced} if for every non-empty proper subset $Q \subsetneq [m]$ we have that
\begin{equation}
	\frac{|Q|-1}{|Q|-r_Q-1} < m_1(A).
\end{equation}

The following lemma now develops the notion of an \emph{induced submatrix} originally introduced (though not explicitly referred to as such) by Rödl and Ruci\'nski~\cite{RR97} for partition regular matrices. Their proofs, adapted for the full generality of abundant matrices and the inhomogeneous case, are included here for completeness.
	
\begin{lemma} \label{lemma:subsystemobservations}
	For every matrix $A \in \ZZ^{r\times m}$ and set of column indices $Q \subseteq [m]$ satisfying $r_Q  > 0$ there exists an invertible matrix $P \in \ZZ^{r\times r}$ such that the submatrix
	\begin{equation}
		(P \! \cdot \! A)^Q_{[r_Q]} = B(P,A,Q) = B \in \ZZ^{r_Q\times|Q|}
	\end{equation}
	is of rank $r_Q$ while the submatrix $(P \! \cdot \! A)^{\widebar{Q}}_{[r_Q]}$ is of rank $0$. For every such $P$ the following hold:
	\begin{itemize} \setlength\itemsep{0em}
		\item[(i)] We have $\rank{(P \!\cdot\! A)^{\widebar{Q}}_{[r]\setminus [r_Q]}} = \mathrm{rk}(A)- r_Q$.
		\item[(ii)] If $A$ is abundant then $B$ is abundant.
		\item[(iii)] For any vector $\bb \in \ZZ^r$  and any solution $\bx \in S(A, \bb)$ we have that $\bx^Q \in S(B, \bc)$, where $\bc = \bc(P,A,Q, \bb) =  (P \cdot \bb)_{[r_Q]} \in \ZZ^{r_Q}$. In particular, if $A$ is positive then so is $B$.
		\item[(iv)] For any $Q' \subseteq \{1, ..., |Q| \}$ there exists $Q'' \subseteq Q$ such that $|Q''| = |Q'|$ and $r_{Q''}(A) = r_{Q'}(B)$.
	\end{itemize}
\end{lemma}

Note that, for each $A$ and $Q$ satisfying these properties, there can
of course exist multiple $P$, but for the remainder of the paper we will fix one particular such $P = P(A,Q)$ and denote $B(P,A,Q)$ by $B(A,Q)$ as well as $\bc(P,A,Q,\bb)$ by $\bc(A,Q, \bb)$. The following block decomposition demonstrates the situation for $Q = \{1, \ldots , |Q|\}$:
	\begin{equation} \label{eq:subsystemillustration}
		P \!\cdot\! A = \left( \! \! \! \begin{array}{c}
			\begin{array}{cc} B & \mathbf{0} \\ 
				X & Y 
			\end{array}
			\end{array} \! \! \! \right)
			\begin{array}{l}
				\big]\ r_Q \textcolor{white}{\big|} \\
				\big]\ r - r_Q \textcolor{white}{\big|} 
			\end{array}
	\end{equation}

\begin{proof}
	We construct $P$ by standard Gaussian elimination, using elementary row operations. We denote the rows of $A$ by $\ba_1, \dots, \ba_r$. Among the rows $\ba_1^{\widebar{Q}}, \dots, \ba_r^{\widebar{Q}}$ of $A^{\widebar{Q}}$ we choose $\rank{A^{\widebar{Q}}}$ linearly independent vectors and express each of the remaining $r - \rank{A^{\widebar{Q}}}$ vectors as a rational linear combination of this basis. Multiplying with the denominators we create integer linear combinations for each of these rows and then we perform the corresponding elementary row operations for each row in $A$. This turns each entry in the $\widebar{Q}$--columns of these $r - \rank{A^{\widebar{Q}}}$ rows into a $0$. Hence the dimension of these rows must be $\rank{A} - \rank{A^{\widebar{Q}}} = r_Q$. We identify a set of $r_Q$ linearly independent rows and permute them to the top of the matrix, hence obtaining the promised block decomposition (\ref{eq:subsystemillustration}). Note that the rank of $B$ is $r_Q$ by construction.

	To prove \emph{(i)}, note that $\rank{(P \!\cdot\! A)^{\widebar{Q}}}_{[r]\setminus [r_Q]}) =  \rank{(P \!\cdot\! A)^{\widebar{Q}}} = \rank{A^{\widebar{Q}}} = \rank{A} -r_Q$ by the definition of $r_Q$.

	By \emph{(i)}, we now have $\rank{A}= \rank{P \!\cdot\! A}= \rank{B} + \rank{A^{\widebar{Q}}}$, so if deleting some two columns of $B$ decreases its rank then deleting the same columns of $P \!\cdot\! A$ decreases its rank. Hence if $B$ is not abundant then $P \!\cdot\! A$ and consequently also $A$ are not abundant. Thus \emph{(ii)} follows.

	To prove \emph{(iii),} note that from (\ref{eq:subsystemillustration}) it follows that for any solution $\bx \in S(A,\bb)= S(P \!\cdot\! A, P \cdot \bb)$ we have that $(P \cdot \bb)_{[r_Q]}  = (P \!\cdot\! (A \cdot \bx) )_{[r_Q]} =  (P \!\cdot\! A)_{[r_Q]} \cdot \bx = B \cdot \bold{x}^Q$, since $(P \!\cdot\! A)^Q_{[r_Q]} = B$ and $(P \!\cdot\! A)^{\widebar{Q}}_{[r_Q]}$ is the $0$--matrix. Therefore $\bold{x}^Q \in S(B, \bc)$. The second statement follows by noting that $\bc(P,A,Q,\bb) = \bold{0}$ provided that $\bb = \bold{0}$.
	
	Lastly, for \emph{(iv)},	let us assume without loss of generality that the columns are permuted such that $Q = \{1, \dots, |Q|\}$ so that we may simply choose $Q'' = Q'$. From \emph{(i)} we know that we can choose a basis of the vectors space generated by the rows of $A$ that consists of $r_Q$ rows $\bold{r}_1, \dots, \bold{r}_{r_Q}$ from $(P \!\cdot\! A)_{[r_Q]}$ and $\rank{A^{\widebar{Q}}}$ rows $\bold{r}_{r_Q+1}, \dots, \bold{r}_{\rank{A}}$ from $(P \!\cdot\! A)_{[r] \backslash [r_Q]}$.	 By construction the vectors $\bold{r}_{r_Q+1}^{\widebar{Q}}, \dots, \bold{r}_{\rank{A}}^{\widebar{Q}}$ are linearly independent from each other, so since $Q'' \subseteq Q$ the vectors $\bold{r}_{r_Q+1}^{\widebar{Q''}}, \dots, \bold{r}_{\rank{A}}^{\widebar{Q''}}$ are as well, implying $\rank{(P \!\cdot\! A)^{\widebar{Q''}}_{[r] \backslash [r_Q]}} = \rank{(P \!\cdot\! A)^{\widebar{Q}}}$. Note that again by construction any linear combination of the vectors $\bold{r}_{1}^{\widebar{Q''}}, \dots, \bold{r}_{r_Q}^{\widebar{Q''}}$ has the last $|\widebar{Q}|$ entries equal to zero and hence cannot be expressed as a linear combination of $\bold{r}_{r_Q+1}^{\widebar{Q''}}, \dots, \bold{r}_{\rank{A}}^{\widebar{Q''}}$, as $\bold{r}_{r_Q+1}^{\widebar{Q}}, \dots, \bold{r}_{\rank{A}}^{\widebar{Q}}$ were linearly independent. It follows that we can add $\rank{(P \!\cdot\! A)^{\widebar{Q''}}_{[r_Q]}}= \rank{B^{\widebar{Q'}}}$ linearly independent vectors from $\bold{r}_{1}^{\widebar{Q''}}, \dots, \bold{r}_{r_Q}^{\widebar{Q''}}$ to the $\rank{(P \!\cdot\! A)^{\widebar{Q}}} = \rank{A^{\widebar{Q}}}$ linearly independent vectors $\bold{r}_{r_Q+1}^{\widebar{Q''}}, \dots, \bold{r}_{\rank{A}}^{\widebar{Q''}}$ to form a basis of the row space of $\rank{A^{\widebar{Q''}}} $. This implies that $\rank{A^{\widebar{Q''}}} = \rank{A^{\widebar{Q}}} + \rank{B^{\widebar{Q'}}}$ from which we can conclude that
	\begin{equation*}
		r_{Q''}(A) = \rank{A} - \rank{A^{\widebar{Q''}}} = r_Q + \rank{A^{\widebar{Q}}} - \rank{A^{\widebar{Q''}}} = \rank{B} - \rank{B^{\widebar{Q'}}} = r_{Q'}(B)
	\end{equation*}
	as desired.
\end{proof}

The following corollary to this lemma will allow us to handle the case of Breaker's strategy when the given matrix is not strictly balanced.

\begin{corollary} \label{cor:strictsubsystem}
	If $A \in \ZZ^{r \times m}$ is positive and abundant then there exists some non-empty set of column indices $Q \subseteq [m]$ such that $B = B(A,Q)$ is abundant, positive, strictly balanced and satisfies $m_1(B) = m_1(A)$. Furthermore, for $\bc = \bc(A,Q,\bb)$ any subset $T \subseteq \NN$ such that $S_0(B,\bc) \cap T^m = \emptyset$ also satisfies $S_0(A,\bb) \cap T^m = \emptyset$.
\end{corollary}

\begin{proof}
	Choose $Q \subseteq [m]$ such that $(|Q|-1)/(|Q|-r_Q-1) = m_1(A)$ and $|Q|$ is minimal with this property. By~\emph{(ii)} and~\emph{(iii)} we know that $B$ is abundant and positive. Assume that there exists $Q' \subsetneq \{ 1, \dots, |Q| \}$ such that $(|Q'|-1)/(|Q'|-r_{Q'}(B)-1) \geq (|Q|-1)/(|Q|-r_Q-1)$. By \emph{(iv)} there must exist $Q'' \subseteq [m]$ with $|Q''| = |Q'| < |Q|$ such that $r_{Q''}(A) = r_{Q'}(B)$. It follows that $(|Q''|-1)/(|Q''|-r_{Q''}(A)-1) \geq m_1(A)$, giving us a contradiction to our choice of $Q$. Finally, the last statement readily follows from \emph{(iii)}.
\end{proof}

The following lemma now establishes some results regarding the rank of induced submatrices of abundant matrices. It also verifies that the maximum $1$--density parameter given in the introduction is indeed well-defined for abundant matrices. Rödl and Ruci\'nski~\cite{RR97} verified this for partition regular matrices. Here we provide a proof for abundant matrices.

\begin{lemma} \label{lemma:welldefined}
	For any abundant matrix $A \in \ZZ^{r\times m}$ and subset of column indices $Q \subseteq [m]$ the following holds. If $|Q| \geq 2$ then $|Q| - r_Q - 1 > 0$. If $|Q| \leq 2$ then $r_Q = 0$.
\end{lemma}

\begin{proof}
	By the Lemma \ref{lemma:subsystemobservations}, $B(A,Q)$ is abundant, has rank $r_Q$ and the number of its columns is $|Q|$. The first statement now follows, since for any abundant matrix the number of columns must be at least two more than its rank. Otherwise deleting any two columns the number of columns, and hence also the rank, would be strictly less than the old rank.

	If $|Q| \leq 2$ then, since $A$ is abundant, deleting the columns in $Q$ does not reduce the rank of $A$. Hence $\rank{A^{\widebar{Q}}} = \rank{A}$ and therefore $r_Q = 0$.
\end{proof}

\subsection{Proof of Rado Games statements}

Let us state some general observations regarding the distribution of edges in the hypergraph $\mS_0(A,\bb,n)$ that will be used in applying both Maker's and Breaker's criterion.

\begin{lemma} \label{lemma:vdWdensity}
	For every positive and abundant matrix $A \in \ZZ^{r\times m}$ and vector $\bb \in \ZZ^r$ for which $\mS(A,\bb) \neq \emptyset$, we have
	\begin{equation*}
		\density{\mS_0(A,\bb,n)} = \Theta \big( n^{m-\rank{A}-1} \big).
	\end{equation*}
\end{lemma}

\begin{proof}
	We observe that each edge in $\mS_0(A,\bb,n)$ can stem from at most $m!$ solutions in $S_0(A,\bb) \cap [n]^m$, so that we have
	\begin{equation*}
		| \, S_0(A,\bb) \cap [n]^m|/m! \leq e(\mS_0(A,\bb,n)) \leq | \, S_0(A,\bb) \cap [n]^m|.
	\end{equation*}
	Using~\eqref{eq:trivialupperbound} and Lemma~\ref{lemma:solutionslowerbound} there therefore exists a constant $c_0 = c_0(A,\bb) > 0$ so that
	\begin{equation*}
		c_0/m! \; n^{m-\rank{A}-1} \leq \density{\mS_0(A,\bb,n)} \leq n^{m-\rank{A}-1}
	\end{equation*}
	giving us the desired statement.
\end{proof}

\begin{lemma} \label{lemma:vdWDeltal}
	For every positive matrix $A \in \ZZ^{r\times m}$, vector $\bb \in \ZZ^r$ and $1 \leq \ell \leq m$ the maximum $\ell$--degree in $\mS_0(A,\bb,n)$ satisfies
	\begin{equation*} \label{eq:MakerHypergraphDegreeUpperbound}
		\Delta_{\ell} (\mS_0(A,\bb,n)) = O \left( \max_{ Q \subseteq [m],\, |Q|= \, \ell} n^{(m-\rank{A})-(|Q|-r_Q)} \right).
	\end{equation*}
\end{lemma}
\begin{proof}
	We have
	\begin{align*}
		\Delta_{\ell} (\mS_0(A,\bb,n)) & \leq \max_{(x_1 , \dots , x_{\ell}) \in  [n]^{\ell}} \big| \{ \bx \in S_0(A,\bb) \cap [n]^m : \exists \, Q \subseteq [m] \text{ s.t. } \bx^{Q} = (x_1, \dots , x_{\ell}) \} \big| \\
		& \leq \binom{m}{\ell} \max_{\substack{(x_1 , \dots , x_{\ell}) \in  [n]^{\ell} \\ Q \subseteq [m] , |Q| = \ell}} \big| \{ \bx \in [n]^{m-\ell} : A^{\widebar{Q}} \cdot \bx^T = \bb - A^{Q} \cdot (x_1, \dots, x_{\ell})^T \} \big| \\
		& \leq m^{\ell} \max_{\substack{Q \subseteq [m] \\ |Q| = \ell}} \, \max_{\bb' \in \ZZ^r} \, \big| \, S(A^{\widebar{Q}},\bb') \cap [n]^{m-\ell} \big|.
	\end{align*}
	Using~\eqref{eq:trivialupperbound} as well as the fact that $|\widebar{Q}| = m - |Q|$ and $r_Q = \rank{A} - \rank{A^{\widebar{Q}}}$, it follows that
	\begin{equation*}
		\Delta_{\ell} (\mS_0(A,\bb,n)) \leq  m^{\ell} \max_{Q \subseteq [m], \; |Q|= \, \ell} \!\!\!\!\! n^{|\widebar{Q}| - \rank{A^{\widebar{Q}}}} = m^{\ell} \max_{Q \subseteq [m], \; |Q|= \, \ell} \!\!\!\!\! n^{(m-\rank{A})-(|Q|-r_Q)}
	\end{equation*}
	giving us the desired statement.
\end{proof}

Using these results we are now ready to provide a proof of Theorem~\ref{thm:ThresholdGeneralizedVdWGames-Proper}. This will be immediately followed by proofs of Proposition~\ref{prop:ThresholdEasyGeneralizedVdWGames} and Corollary~\ref{cor:ThresholdGeneralizedVdWGames}.

\begin{proof}[Proof of Theorem~\ref{thm:ThresholdGeneralizedVdWGames-Proper}]

We will prove that the threshold bias satisfies $q(\mS_0(A,\bb,n)) = \Theta(n^{1/m_1(A)})$ by showing that the criteria of Theorem~\ref{cor:MakerWinCriterion} are met by $\mS_0(A,\bb,n)$ and that the criteria of Theorem~\ref{cor:BreakerWinCriterion} are met by $\mS_0(\Sub{A}{Q},\Sub{\bb}{Q},n)$ where $Q \subseteq [m]$ will be some appropriately chosen set of column indices. The bounds on the threshold bias obtained this way will asymptotically be the same, giving the desired statement.

\medskip

\noindent {\bf Maker's Strategy.} Since $A$ is abundant we know by Lemma~\ref{lemma:welldefined} that $r_Q = 0$ for any $Q \subseteq [m]$ satisfying $|Q| \leq 2$. It follows that by Lemma~\ref{lemma:vdWDeltal} we have $\Delta_1 (\mS_0(A,\bb,n)) = O ( n^{(m-\rank{A})-1} )$ and $\Delta_2 (\mS_0(A,\bb,n)) = O ( n^{(m-\rank{A})-2} )$. Lemma~\ref{lemma:vdWdensity} therefore immediately implies that both Condition~\emph{(M1)} and Condition~\emph{(M2)} hold. As $v(\mS_0(A,\bb,n)) = n$ and $\rank{A} \geq 1$ Lemma~\ref{lemma:vdWdensity} also implies that $\density{\mS_0(A,\bb,n)} = o(n^{m-1})$, so that Condition~\emph{(M3)} holds as well. It follows that Theorem~\ref{cor:MakerWinCriterion} applies and establishes the desired lower bound on the threshold bias since
\begin{align*}
	\min_{2 \leq \ell \leq m} \left( \frac{\density{\mS_0(A,\bb,n)}}{\Delta_{\ell} (\mS_0(A,\bb,n))} \right)^{\frac{1}{\ell-1}}
	& =  \min_{2 \leq \ell \leq m} \left( \frac{\Theta \big( n^{m-\rank{A}-1} \big)}{O \big( \max_{\substack{ Q \subseteq [m] \\ |Q| = \, \ell}} n^{(m-\rank{A})-(|Q|-r_Q)} \big)} \right)^{\frac{1}{\ell-1}} \\
	& = \min_{\substack{Q \subseteq [m] \\ |Q| \geq 2}} \Omega \left( n^{\frac{|Q| - r_Q - 1}{|Q|-1}} \right) = \Omega \big( n^{1/m_1(A)} \big).
\end{align*}

\medskip

\noindent {\bf Breaker's Strategy.} Let $Q$ and the corresponding $B$ and $\bc$ be as given by Corollary~\ref{cor:strictsubsystem}. Considering Maker's choice of vertices as the set $T$ in this corollary, it is clear that if Breaker can keep Maker from covering any solution in $S_0(B,\bc)$ then Maker must also fail at covering any solution in $S_0(A,\bb)$. For this part, we can therefore without loss of generality assume that $A$ is not just positive and abundant, but also strictly balanced, that is $(m-1)/(m-\rank{A}-1) = m_1(A)$ as well as $(|Q|-1)/(|Q|-r_Q-1) < m_1(A)$ for any $Q \subsetneq [m]$ satisfying $|Q| \geq 2$.

Note that $v(\mS_0(A,\bb,n)) = n$ so that we clearly have $v(\mS_0(A,\bb,n)) \to \infty$. Since $A$ is abundant we know by Lemma~\ref{lemma:welldefined} that $r_Q = 0$ for any $Q \subseteq [m]$ satisfying $|Q| = 1$. Lemma~\ref{lemma:vdWDeltal} as well as Lemma~\ref{lemma:vdWdensity} combined with the fact that $\Delta_1(\mS_0(A,\bb,n)) \geq \density{\mS_0(A,\bb,n)}$ therefore imply that
\begin{equation} \label{eq:vdwDelta1}
	\Delta_1(\mS_0(A,\bb,n))^{\frac{1}{m-1}} = \Theta \big( n^{m-\rank{A}-1} \big)^{\frac{1}{m-1}} = \Theta \big( n^{1/m_1(A)} \big).
\end{equation}
In the last step we have used the assumption that $A$ is strictly balanced. That same assumption also states that for all $Q \subseteq [m]$ satisfying $2 \leq |Q| < m$ we have $(|Q|-1)/(|Q|-r_Q-1) < m_1(A)$ so that
\begin{align*}
	\frac{(m-\rank{A})-(|Q|-r_Q)}{m-|Q|} & = \frac{(m-\rank{A}-1)-(|Q|-r_Q-1)}{(m-1)-(|Q|-1)} \\
	& < \frac{(m-\rank{A}-1)-(|Q|-1) / m_1(A)}{(m-1)-(|Q|-1)} \\
	& = \frac{1}{m_1(A)} \, \frac{1 - (|Q|-1)/(m-1)}{1 - (|Q|-1)(m-1)} = \frac{1}{m_1(A)}.
\end{align*}

Using this, it follows that there exists some $\epsilon = \epsilon(A) > 0$ so that for any $2 \leq \ell \leq m$ we have by Lemma~\ref{lemma:vdWDeltal} and~\eqref{eq:vdwDelta1} that
\begin{align*}
	\Delta_{\ell}(\mS_0(A,\bb,n))^{\frac{1}{m-\ell}} \, n^{\epsilon} & = O \Big( \max_{Q \subseteq [m], \; |Q| = \, \ell} \!\!\!\!\! n^{(m-\rank{A})-(|Q|-r_Q)} \Big)^{\frac{1}{m-\ell}} \, n^{\epsilon} \\
	& = O \Big( \max_{Q \subseteq [m], \; |Q| = \, \ell} \!\!\!\!\! n^{\frac{(m-\rank{A})-(|Q|-r_Q)}{m-|Q|} + \epsilon} \Big) \\
	& = O \big( n^{1/m_1(A)} \big) = O \big( \Delta_1(\mS_0(A,\bb,n))^{\frac{1}{m-1}} \big).
\end{align*}
It follows that Theorem~\ref{cor:BreakerWinCriterion} applies and due to~\eqref{eq:vdwDelta1} establishes the desired upper bound on the threshold bias.
\end{proof}

To conclude this part, let us prove Proposition~\ref{prop:ThresholdEasyGeneralizedVdWGames}.

\begin{proof}[Proof of Proposition~\ref{prop:ThresholdEasyGeneralizedVdWGames}]

	Let us start by noting that if $S(A,\bb) = \emptyset$, then the game hypergraph $\mS_0(A,\bb,n)$ is empty and the game trivially is an immediate win for Breaker.

	Let us therefore consider the case where $S(A,\bb) \neq \emptyset$ and $A$ is positive and non-abundant but still satisfies $\rank{A} > 0$. From the non-abundancy, it follows that there exist two column indices $1 \leq i_1,i_2 \leq m$ such for $Q = [m] \setminus \{i_1,i_2\}$ we get $\rank{A^Q} < \rank{A}$. It follows that there exist a set of basic row transformations and a row index $j$ such that row $j$ of $A^Q$ consists only of $0$ entries while row $j$ of $A$ can be taken as a basis vector of the space spanned by the rows of $A$. As $A$ is positive, this row of $A$ does not consist of all $0$ entries. If $i_1 = i_2$, then it follows that there that any $\bx = (x_1,\dots,x_m) \in S(A,\bb)$ satisfies $x_{i_1} = 0$, contradicting the assumption that $A$ is positive. If $i_1$ and $i_2$ are distinct, then it follows that there exist $v_1,v_2,b' \in \ZZ$ with $v_1,v_2 \neq 0$ such that any $\bx = (x_1,\dots,x_m) \in S(A,\bb)$ satisfies $v_1 x_{i_1} + v_2 x_{i_2} = b'$. Now, whenever Maker occupies some $i \in [n]$, Breaker can simply pick $(b' - v_1 \, i)/v_2$ and $(b' - v_2 \, i)/v_1$ (if these are indeed integer values in $[n]$) and thus block Maker's ability to cover any solution. It follows that Breaker has a winning strategy with a bias of at most $2$.
\end{proof}

\subsection{Solutions with repeated entries} \label{subsec:repeatedcomponents}

It is apparent already from the matrix of the $k$--AP game that allowing solutions with repeated components might make the game `easier’ for Maker. Indeed, the vector $(z,\ldots , z)\in [n]^k$ is a solution of the $k$--AP game, hence occupying any one element of $[n]$  would immediately provide Maker with such a solution and a win in its first move.

On the other hand constant vectors are the only non-proper solutions of the $k$--AP matrix, so the $k$--AP game with proper solutions does not become any different even if we allowed equality of any combination of the coordinates except for all of them.

For the matrix $(1\,\,\,1\,\,-1)$, the solutions of which are called Schur triples, allowing any combination of coordinates to be equal does not make the game significantly easier for Maker. Indeed, there are no positive solutions with $x_1=x_3$ or $x_2=x_3$ and to block solutions  with $x_1=x_2$ Breaker only needs at most two extra moves in each round, since he might need to occupy the double and the half of  Maker's previous move.

For another classic equation, the Sidon equation $x_1 + x_2 = x_3 + x_4$, the result of the game changes greatly according to which combination of coordinates we allow to be equal. According to Theorem~\ref{thm:ThresholdGeneralizedVdWGames-Proper} the game with proper solutions has threshold of the order $n^{2/3}$. If we allowed $x_1=x_3$ and $x_2=x_4$, then occupying any two different integers would provide Maker with a win, so the threshold bias would grow to $n-1$. If we were to allow for Maker solutions with repeated coordinates $x_1=x_2$, (but required $x_3\neq x_1, x_1\neq x_4, x_4 \neq x_3$), then it turns out that the game's threshold bias is the same order of magnitude as the one of the game with proper solutions.

Subsequently we will be after identifying exactly which component-equalities make the game easier for Maker and which ones do not. More precisely which one of them change the order of the threshold bias compared to the bias $q(\mS_0(A,\bb,n))$ and which ones do not. Identifying this correct notion of `non-degenerate’ solution for our setup takes a few definitions.

Given a solution $\bx = (x_1, \dots, x_m) \in S(A,\bb)$ for an integer-valued matrix $A\in \ZZ^{r \times m}$ and vector $\bb \in \ZZ^r$, let

\begin{equation*}
	\fp[\bx] = \big\{ \{ 1 \leq j \leq m : x_i = x_j \} : 1 \leq i \leq m \big\}
\end{equation*}
denote the set partition of the column indices $[m]$ indicating the repeated entries in $\bx$. Note that for $\bx \in S_0(A,\bb)$ we have $\fp[\bx] = \{ \{1\}, \dots, \{m\} \}$. Given some set partition $\fp$ of $\{1, \dots ,m\}$, let $A_{\fp}$ denote the matrix obtained by summing up the columns of $A$ according to $\fp$, that is for $\fp = \{ T_1, \dots, T_s \}$ such that $\min (T_1) < \dots < \min(T_s)$ for some $1 \leq s \leq m$ and $\bc_i$ the $i$--th column vector of $A$ for every $1 \leq i \leq m$, we have
\begin{equation*}
	A_{\fp} = \left( \begin{array}{ccccccc}
		\sum\limits_{i \in T_1} \! \bc_i & \Big| & \sum\limits_{i \in T_2} \! \bc_i & \Big| & \cdots & \Big| & \sum\limits_{i \in T_s} \! \bc_i
 	\end{array} \right).
\end{equation*}
Note that the assumption $\min (T_1) < \dots < \min(T_s)$ ensures that
this notion is well-defined and that $A_{\fp} = A$ for $\fp = \{ \{ 1
\}, \dots, \{ m \} \}$.

Using these definitions we can now define when a solution is considered to be non-degenerate:
\begin{enumerate}
	\item If $A$ is positive and abundant, then a solution $\bx \in S(A,\bb) \cap \NN^m$ is defined to be \emph{non-degenerate} if $|\fp[\bx]| \geq 2$ and $A_{\fp}$ is either non-abundant or it is abundant and satisfies $m_1(A_{\fp}) \geq m_1(A)$.
	\item If $A$ is positive and non-abundant, then a solution $\bx \in S(A,\bb) \cap \NN^m$ is defined to be \emph{non-degenerate} if $|\fp[\bx]| \geq 2$.
\end{enumerate}

For example for the (positive and abundant) matrices associated to $k$--APs and Schur triple the only non-degenerate solutions are the proper ones. For the matrix associated with the Sidon equation $x_1 + x_2 = x_3 + x_4$,  $3$--APs are non-degenerate solutions with repeated entries.

The main result of this section shows that this is the right definition for those solutions, which do not make the game any easier for Maker (i.e., do not lose any of the `complexity’ of the original system due to coordinate repetition). Note that our definition includes solutions $\bx \in S(A,\bb)$ for which $\rank{A_{\fp[\bx]}} = \rank{A}$. These were called \emph{non-trivial} by Ru\'e et al.~\cite{RSZ15}.  Both definitions extend a previous definition for single-line equations due to Ruzsa~\cite{Rz93}. We now let
\begin{equation*}
	S_1(A,\bb) = \big\{ \bx \in S(A,\bb) : \bx \text{ is non-degenerate} \big\}
\end{equation*}
and remark that $S(A,\bb) \supseteq S_1(A,\bb) \supseteq S_0(A,\bb)$. Furthermore we denote by $\mS_1(A,\bb,n)$ the hypergraph containing all non-degenerate solutions in $[n]$, that is
\begin{equation*}
	\mS_1(A,\bb,n) = \big\{ \{ x_1 , \dots , x_m \} : (x_1,\dots , x_m) \in \mS_1(A,\bb) \cap [n]^m \big\}.
\end{equation*}
Note that $\mS_1(A,\bb,n)$ in contrast to $\mS_0(A,\bb,n)$ is not necessarily uniform.

The following result can be proven as a corollary to Theorem~\ref{thm:ThresholdGeneralizedVdWGames-Proper} and shows that allowing non-degenerate solutions for Maker does not change the order of the threshold bias compared to the proper game. In other words, with only a constant factor times the original threshold bias, Breaker is able to block not just all proper solutions but also every non-degenerate solution. Here we let the notation $\Theta(0)$ mean $\Theta(1)$.
\begin{corollary} \label{cor:ThresholdGeneralizedVdWGames}
	For every positive matrix $A \in \ZZ^{r \times m}$ and vector $\bb \in \ZZ^r$ the threshold bias of the Maker-Breaker $(A,\bb)$--game on $[n]$ allowing non-degenerate solutions satisfies $q(\mS_1(A,\bb,n)) = \Theta \left( q(\mS_0(A,\bb,n)) \right)$.
\end{corollary}

\begin{proof}
	The central observation necessary to prove this corollary is that for all non-degenerate partitions $\fp$ (and in fact for all non-vacant partitions) we have 
	\begin{equation}
		\big\{ \{ x_1, \dots, x_m \} : (x_1,\dots,x_m) \in S(A,\bb) \cap [n]^m \text{ and } \fp[\bx] = \fp \big\} = \mS_0(A_{\fp},\bb,n).	
	\end{equation}
	Now if $A$ is positive and abundant then for non-degenerate $\fp$ the bias threshold of the game played on $\mS_0(A_{\fp},\bb,n)$ either satisfies $q(\mS_0(A_{\fp},\bb,n)) \leq 2$ by Proposition~\ref{prop:ThresholdEasyGeneralizedVdWGames} if $A_{\fp}$ is non-abundant or it satisfies $q(\mS_0(A_{\fp},\bb,n)) = \Theta(n^{1/m_1(A_{\fp})}) = O(n^{1/m_1(A)})$ by Theorem~\ref{thm:ThresholdGeneralizedVdWGames-Proper} if $A_{\fp}$ is abundant since we required that $m_1(A_{\fp}) \geq m_1(A)$. Noting that the number of possible non-degenerate partitions of $[m]$ is clearly bounded from above by $m!$ gives the desired result through Breaker's possibility to use strategy splitting, see Remark~\ref{rmk:biasdoubling}.
	
	If $A$ is non-abundant, then for any non-degenerate $\fp$ the matrix $A_{\fp}$ is also non-abundant and hence $q(\mS_0(A_{\fp},\bb,,n)) = q(\mS_0(A,\bb,n)) = \Theta(1)$.
\end{proof}

Let us motivate why this is the `right’ notion of non-degenerate solutions. Observe that given some partition $\fp$ of $[m]$ the set $\{ \bx \in S(A,\bb) : \fp[\bx] = \fp\}$ is either empty or it trivially consists only of non-degenerate or only of degenerate solutions. We therefore respectively also refer to $\fp$ as either \emph{vacant}, \emph{non-degenerate} or \emph{degenerate}. We will later remark in Subsection~\ref{subsec:morerepeatedcomponents} that allowing Maker to also win by occupying any solution belonging to a fixed degenerate partition $\fp$ does change the order of the threshold bias. In other words, non-degenerate solutions indeed provide an exact characterization for classes of solutions with repeated components that do not change the complexity of the original linear homogenous system.

\section{Proof of Theorem~\ref{thm:ThresholdHypergraphGames} -- Small Hypergraph Games} \label{sec:HypergraphGames}

First, observe that if $\mG$ is a collection of $e(\mG)$ independent edges, then Maker has a winning strategy if $q< {n-r(e(\mG)-1 ) \choose r}/(e(\mG)-1)$. So we may assume that this is not the case. We recall that $\mH(\mG,n)$ was the hypergraph of all copies of $\mG$ in $\mK_n^{(r)}$. We observe that $\mH(\mG,n)$ is $e(\mG)$--uniform and clearly satisfies
\begin{equation} \label{eq:hypergraphvertexnr}
	v(\mH(\mG,n))= \binom{n}{r} = \Theta(n^r)
\end{equation}
as well as $e(\mH(\mG,n)) = \binom{n}{v(\mG)} \, v(\mG)! \, / \, \text{aut}(\mG) = \Theta(n^{v(\mG)})$. In particular, it follows that
\begin{equation} \label{eq:hypergraphdensity}
	\density{\mH(\mG,n)} = \Theta(n^{v(\mG)-r}).
\end{equation}
Lastly observe that for $1 \leq \ell \leq e(\mG)$ we have
\begin{equation} \label{eq:hypergraphmaxldegree}
	\Delta_{\ell} (\mH(\mG,n)) = \Theta \left( \! \max_{\mF \subseteq \mG,\, e(\mF) = \ell} n^{v(\mG) - v(\mF)} \right).
\end{equation}

We will now prove that the threshold bias satisfies $q(\mH(\mG,n)) = \Theta(n^{1/m_r(\mG)})$ by showing that the criteria of Theorem~\ref{cor:MakerWinCriterion} are met by $\mH(\mG,n)$ and that the criteria of Theorem~\ref{cor:BreakerWinCriterion} are met by $\mH(\mF,n)$ where $\mF$ will be some appropriate dense subgraph of $\mG$. The bounds on the threshold bias obtained this way will asymptotically be the same, giving the desired statement.

\medskip

\noindent {\bf Maker's Strategy.}~\eqref{eq:hypergraphmaxldegree} implies that $\Delta_1(\mH(\mG,n)) = \Theta (n^{v(G)-r})$ as  well as $\Delta_2(\mH(\mG,n)) = O (n^{v(G) - (r+1)})$ so that Conditions~\emph{(M1)} and~\emph{(M2)} immediately follow from~\eqref{eq:hypergraphdensity}. Now, as we have already excluded the case that $\mG$ is a collection of $e(\mG)$ independent edges, we have $v(\mG)/e(\mG) < r$ so that $v(\mG)-r < r \, (e(\mG)-1)$ and hence Condition~\emph{(M3)} is satisfied by Equations~\eqref{eq:hypergraphvertexnr} and~\eqref{eq:hypergraphdensity}. It follows that Theorem~\ref{cor:MakerWinCriterion} applies and establishes the desired lower bound on the threshold bias since
\begin{align*}
	\min_{2 \leq \ell \leq e(\mG)} \left( \frac{\density{\mH(\mG,n)}}{\Delta_{\ell} (\mH(\mG,n))} \right)^{\frac{1}{\ell-1}}
	& =  \min_{2 \leq \ell \leq e(\mG)} \left( \frac{\Theta \big( n^{v(\mG) - r} \big) }{\Theta \big( \max_{\mF \subseteq \mG,\, e(\mF) = \ell} n^{v(\mG) - v(\mF)} \big) } \right)^{\frac{1}{\ell-1}} \\
	& =  \min_{\mF \subseteq \mG,\, e(\mF) \geq 2} \Theta \left( n^{\frac{v(\mF) - r}{e(\mF)-1}} \right) =  \Theta \big( n^{1/m_r(\mG)} \big).
\end{align*}

\medskip

\noindent {\bf Breaker's Strategy.} Note that we can restrict our attention to the case in which $\mG$ is strictly $r$--balanced, as otherwise we can replace $\mG$ with a strictly $r$--balanced subhypergraph $\mF \subset \mG$. Indeed, if Breaker can keep Maker from occupying $\mF$, then he clearly also succeeds in keeping Maker from occupying a copy of $\mG$. So we may assume that $m_r(\mG) = (e(\mG)-1)/(v(\mG)-r)$ and that $m_r(\mF) = (e(\mF)-1)/(v(\mF)-r) < m_r(\mG)$ for all subgraphs $\mF \subsetneq \mG$ on at least $r+1$ vertices.

Clearly $v(\mH(\mG,n)) \to \infty$ by~\eqref{eq:hypergraphvertexnr}. We note that by~\eqref{eq:hypergraphmaxldegree} as well as the assumption that $\mG$ is strictly $r$--balanced we have
\begin{equation} \label{eq:Delta1threshold}
	\Delta_1 (\mH(\mG,n))^{\frac{1}{e(\mG)-1}} = \Theta \big( n^{\frac{v(\mG)-r}{e(\mG)-1}} \big) = \Theta \big( n^{1/m_r(\mG)} \big).
\end{equation}
Since $\mG$ is strictly $r$--balanced, we also have that for every $2 \leq \ell \leq e(\mG)$ and every subhypergraph $\mF \subset \mG$ with $e(\mF) = \ell$ edges
\begin{align*}
	\frac{v(\mG)-v(\mF)}{e(\mG)-\ell} & = \frac{(v(\mG)-r)-(v(\mF)-r)}{(e(\mG)-1)-(e(\mF-1))} = \frac{1}{m_r(\mF)}\; \frac{1- \frac{v(\mF)-r}{v(\mG)-r}}{1- \frac{e(\mF)-1}{e(\mG)-1}} < \frac{1}{m_r(\mG)}.
\end{align*}
Therefore there exists a sufficiently small $\epsilon = \epsilon(r,\mG)$ such that for any $2 \leq \ell \leq e(\mG)$ we have by Equations~\eqref{eq:hypergraphvertexnr},~\eqref{eq:hypergraphmaxldegree} and~\eqref{eq:Delta1threshold} that
\begin{equation*}
	\Delta_{\ell}(\mH(\mG,n))^{\frac{1}{e(\mG)-\ell}} \; v(\mH(\mG,n))^{\epsilon} = \Theta \left( \max_{\substack{\mF \subseteq \mG \\ e(\mF) = \ell}} n^{\frac{v(\mG)-v(\mF)}{e(\mG)-\ell} + r \epsilon} \right) = O \left( n^{\frac{1}{m_r(\mG)}} \right) = O \left( \Delta_1 (\mH(\mG,n))^{\frac{1}{e(\mG)-1}} \right).
\end{equation*}
It follows that Theorem~\ref{cor:BreakerWinCriterion} applies and due to~\eqref{eq:Delta1threshold} establishes the desired upper bound on the threshold bias. \hfill $\square$

\section{Concluding Remarks} \label{sec:conclusion}

In this paper, we have established general criteria for hypergraphs $\mH$, which guarantee that the uniformly random Maker-strategy is essentially optimal in the biased Maker-Breaker game on $\mH$. We have proved that several natural games fall into this category. This included Rado games for solutions of linear equation systems as well as $\mG$--building games for any fixed uniform hypergraph $\mG$.

\subsection{Combining the Criteria for Maker and Breaker}

We note that one can easily combine Theorem~\ref{cor:MakerWinCriterion} and Theorem~\ref{cor:BreakerWinCriterion} to form the following statement giving the exact asymptotic behaviour of the bias threshold for games with a hypergraph that is not dense, roughly regular, and has an appropriate separation of the $\ell$--degrees from the degrees.

\begin{corollary}
	For every $k \geq 2$ the following holds. If $\mH = (\mH_n)_{n
          \in \NN}$ is a sequence of $k$--uniform hypergraphs for which
        there exists an $\epsilon > 0$ so that we have
	\begin{itemize}
		\item[\text{(I)}] \enspace $\density{\mH_n} = o \big(v(\mH_n)^{k-1} \big),$
		\item[\text{(II)}] \enspace $\Delta_1(\mH_n) = O \big( \density{\mH_n} \big),$
		\item[\text{(III)}] \enspace$\Delta_{\ell} (\mH_n)^{\frac{1}{k-\ell}} \, v(\mH_n)^{\epsilon} = O \big( \Delta_1 (\mH_n)^{\frac{1}{k-1}} \big)$ for every $2 \leq \ell \leq k-1$,
	\end{itemize}
	then the threshold biases of the games played on $\mH_n$ satisfy
	\begin{equation}
		q(\mH_n) = \Theta \left( \density{\mH_n}^{\frac{1}{k-1}} \right).
	\end{equation}
\end{corollary}

As the reader of the proofs of Theorem~\ref{thm:ThresholdGeneralizedVdWGames-Proper} and Theorem~\ref{thm:ThresholdHypergraphGames} will have noticed, this would only be applicable in the special case where the matrix or the hypergraph to be built is strictly balanced. For the proof of the full statement of these results we needed the two separate statements as well as the argument that without loss of generality one can replace the matrix or the hypergraph with a denser substructure when determining a strategy for Breaker.

\subsection{Obtaining constants}

In our main theorems we determine the right order of magnitude of the threshold biases for the games considered. Hence one might rightfully be interested in obtaining more precise statements, involving the constant factors. For the triangle-building game Chv\'atal and Erd\H os~\cite{CE78} established upper and lower bounds that are tight up to a constant factor $\sqrt{2}$. Their upper bound was slightly improved by Balogh and Samotij~\cite{BS11}, however the value of the right constant factor is still outstanding.

We state here a couple of bounds for the $3$--AP game,  where we already  established that the threshold bias is of the order $\sqrt{n}$.

\begin{proposition}
	For the threshold bias $q(n)$ of the 3-AP game played on $[n]$ we have
	\begin{align*}
		\sqrt{\frac{n}{12} - \frac{1}{6}}\leq q(n) \leq \sqrt{3n}.
	\end{align*}
\end{proposition}
\begin{proof}
Let us first prove the upper bound by providing a winning strategy for Breaker if he is given a bias of $q \geq \sqrt{3n}$. The strategy will simply consist of blocking all possible 3-APs containing Maker's last choice and one of its previous choices. As for each fixed pair of integers there are at most three 3-APs containing them and Maker occupies at most $M=\lceil n/(q+1)\rceil$ integers during the course of the whole game, the number of 3-APs to be blocked is never more than $3 \, (M-1)$. Since
\begin{align*}
	3 \, \left( M-1\right) \leq q	
\end{align*}
for $q \geq \sqrt{3n}$, Breaker has enough moves in each round to occupy the (at most) one unoccupied element in each of the dangerous 3-APs.

For the lower bound we use the generalization of a criterion that was developed by Beck for Maker's win in the unbiased van der Waerden game~\cite{Be81}. He later stated a biased version~\cite{Be08} and this is what we will apply here.

\begin{theorem}[Biased Maker's Win Criterion~\cite{Be08}] \label{thm:BiasedMakerCriterion}
	Let $\mH$ be a hypergraph and $q \in \NN$. Maker has a winning strategy as the first player in the $q$--biased game on $\mH$ if
	\begin{align} \label{eq:BiasedMakerCriterion}
		\sum_{H \in \mH} \left( \frac{1}{1 + q} \right)^{|H|} > \frac{q^2}{(1 + q)^3} \; \Delta_2(\mH) \; v(\mH).
	\end{align}
\end{theorem}

For the hypergraph $\mH_{n}$ of $3$--APs in $[n]$ we observe that $v(\mH_n) = n$, $e(\mH_n) \geq n^2/4 - n/2 $, and $\Delta_2(\mH_n) \leq 3$. Consequently with a bias of $q <\sqrt{\frac{n}{12} - \frac{1}{6}}$ the condition (\ref{eq:BiasedMakerCriterion}) holds for $\mH_n$ and Theorem~\ref{thm:BiasedMakerCriterion} provides the winning strategy for Maker.
\end{proof}

Observe that the constants $\sqrt{1/12}$ and $\sqrt{3}$ are only a factor $6$ apart, it would be interesting to close this gap.
\begin{question}
	Prove the existence of a constant $C>0$, such that the threshold bias of the 3-AP game is $(C+o(1))\sqrt{n}$.
\end{question}

It should be noted that one may also apply Theorem~\ref{thm:BiasedMakerCriterion} to the $k$--AP game and obtain a lower bound of the right order of magnitude on the the threshold bias for every $k \geq 3$. The ad-hoc argument for Breaker's win does not seem to generalize immediately.

\begin{conjecture}
	For every positive and abundant matrix $A \in \ZZ^{r \times m}$ and vector $\bb \in \ZZ^r$,  there exists a constant $C = C(A,\bb) > 0$ such that $q(\mS_0(A,\bb,n)) = (C+o(1)) \, n^{1/m_1(A)}$.
\end{conjecture}

The analogous question for graph-building games has been posed by Bednarska and {\L}uczak~\cite{BL00}. For hypergraph-building games the same question can of course also be asked.

\subsection{More on repeated entries} \label{subsec:morerepeatedcomponents}

It is not necessary to compare games with repeated solutions to the proper game hypergraph $\mS_0(A,\bb,n)$. For each family $\mP$ of non-vacant partitions of $[n]$, one can define the game hypergraph $\mS (A,\bb , n, \mP)$ containing all subsets that consist of the distinct components of such a solution to $A \cdot \bx^T = \bb^T$ for which  $\fp[x] \in \mP$, that is
\begin{equation}
	\mS (A, \bb, n, \mP) = \{ \{x_1,\dots,x_m\} : (x_1,\dots,x_m) = \bx \in S(A,\bb) \cap [n]^m  \text{ and } \fp[\bx] \in \mP \}.
\end{equation}

Theorem~\ref{thm:ThresholdGeneralizedVdWGames-Proper} of course deals with the case where $\mP$ consists just of the partition $\{ \{1\}, \dots, \{m\} \}$ and Corollary~\ref{cor:ThresholdGeneralizedVdWGames} with the case where $\mP = \{ \fp[\bx] : \bx \in S_1(A,\bb) \}$ consists of all non-degenerate partitions.

Considering the proof of Corollary~\ref{cor:ThresholdGeneralizedVdWGames} it should be clear that playing on the hypergraph of all solutions to $A$ with repetitions indicated by some given partition $\fp$ is the same as playing on the hypergraph $\mS_0(A_{\fp},\bb)$. The following statement follows immediately and can therefore be seen both as a generalization of but also an easy corollary to Theorem~\ref{thm:ThresholdGeneralizedVdWGames-Proper}.

\begin{theorem} \label{thm:ThresholdGeneralizedVdWGames-MostGeneral}
	For every matrix $A \in \ZZ^{r\times m}$, vector $\bb \in \ZZ^r$ and family $\mP$ of set partitions of $[m]$ the corresponding threshold bias satisfies $q(\mS(A,\bb, n, \mP)) = \Theta \big( \max_{\fp \in \mP} q(A_{\fp},\bb,n) \big)$.
\end{theorem}

This of course implies that if there exists $\fp \in \mP$ such that $A_{\fp}$ is positive and abundant and $\fp' \in \mP$ is the one of those partitions that minimizes the parameter $m_1(A_{\fp'})$, then we have $q(\mS(A,\bb, n, \mP)) = \Theta \big( n^{1/m_1(A_{\fp'})} \big)$. This result also shows that the notion of non-degeneracy as defined in the introduction is the broadest possible notion that does not change the character of the game, that is the asymptotic behaviour of its bias threshold.

\subsection{The probabilistic intuition} \label{subsec:probintuition}

It was Chv\'atal and Erd\H{o}s who first pointed out some surprising similarities between certain positional games and results in random graphs. Given some hypergraph $\mH$ let the \emph{appearance threshold} $p(\mH)$ be the threshold probability for the property that the random set $V(\mH)_p$ contains an edge. The \emph{probabilistic intuition} states that the appearance threshold $p(\mH)$ hints at the bias threshold $q(\mH)$ of the game $\bG(\mH;q)$, namely that $q(\mH) \sim p(\mH)^{-1}$. This intuition holds true for several ‘global’ properties such as hamiltonicity or connectivity. We have previously remarked that the biased Erd\H{o}s-Selfridge strategy states that Breaker can do at least as well as when both players act randomly.

Obviously for the games studied by Bednarska and {\L}uczak~\cite{BL00} as well as the two types of Maker-Breaker games studied in this paper, this probabilistic intuition fails. The appearance threshold for a given fixed $r$--uniform hypergraph $\mG$ in the $r$--uniform random graph $G_{n,p}^{(r)}$ occurs around $n^{-1/m(\mG)}$ where $m(\mG) \neq m_r(\mG)$ is the density of $\mG$ maximized over all subgraphs. Likewise, the appearance threshold of homogeneous solutions to a given positive and abundant matrix $A \in \ZZ^{r \times m}$ occurs around $n^{-1/m(A)}$ where $m(A) \neq m_1(A)$ is a parameter of $A$ maximized over all induced submatrices, see Ru\'e et al.~\cite{RSZ15}. As an example, $k$--APs start to appear around $n^{-2/k}$ whereas we have shown the threshold bias to satisfy $n^{1/(k-1)}$.

However, as already noted in the introduction, a different type of random intuition still plays an important role. Our results have a strong connection to sparse Tur\'an- and Szemer\'edi-type statements. Given an $r$--uniform hypergraph $\mG$, let $\text{ex}(n,\mG)$ be the largest number of edges in a $\mG$--free subgraph of $\mathcal{K}_n^{(r)}$ and let $\pi_k(\mG) = \lim_{n \to \infty} \text{ex}(n,\mG) / {n \choose r}$. For $\epsilon > 0$ an $r$--uniform hypergraph $\mF$ is called \emph{$(\mG,\epsilon)$--Tur\'an} if every subgraph of $\mF$ with at least $(\pi_r(\mG) +\epsilon) \, e(\mF)$ edges contains a copy of $\mG$. Conlon and Gowers (in the strictly balanced case)~\cite{CG10} and independently Schacht~\cite{Sch12} showed that the threshold probability of the event that $G_{n,p}^{(r)}$ is $(\mG,\epsilon)$--Tur\'an is $\Theta(n^{-1/m_r(\mF)})$. Compare this to our result that the threshold bias of the Maker-Breaker $\mG$--game is $\Theta(n^{1/m_r(\mG)})$. Similarly, Schacht~\cite{Sch12} showed that for a given density regular matrix $A \in \ZZ^{r \times m}$ the threshold probability for the event that $[n]_p$ is $(\delta,A)$--stable is $\Theta(n^{-1/m_1(A)})$. The third author~\cite{Sp16} as well as independently Hancock, Staden and Treglown~\cite{HST17} extended this result to abundant matrices. Our result shows that the threshold bias of the Maker-Breaker $A$--game lies around $\Theta(n^{1/m_1(A)})$ for the much broader class of positive and abundant matrices. We call the intuition one might infer from this \emph{probabilistic Tur\'an intuition} for biased Maker-Breaker games. We have proven two criteria for Breaker as well as Maker that provide some criteria to verify if this intuition indeed holds true for a given game.

\paragraph{Acknowledgements.} We would like to thank Ma{\l}gorzata
Bednarska-Bzd{\c e}ga, Shagnik Das and Tam{\'a}s M\'esz\'aros for clarifying discussions and constructive suggestions. We also thank the referees of the paper for several suggestions and comments.

\bibliography{bib}

\begin{thebibliography}{10}

\bibitem{BS11}
J.~Balogh and W.~Samotij.
\newblock On the {C}hv{\'a}tal-{E}rd{\H{o}}s triangle game.
\newblock {\em The Electronic Journal of Combinatorics}, 18(P72), 2011.

\bibitem{Be81}
J.~Beck.
\newblock Van der {W}aerden and {R}amsey type games.
\newblock {\em Combinatorica}, 1(2):103--116, 1981.

\bibitem{Be85}
J.~Beck.
\newblock Random graphs and positional games on the complete graph.
\newblock {\em North-Holland Mathematics Studies}, 118:7--13, 1985.

\bibitem{Be08}
J.~Beck.
\newblock {\em {Combinatorial Games: {Tic-Tac-Toe} Theory}}.
\newblock Cambridge University Press, 2008.

\bibitem{BL00}
M.~Bednarska and T.~{\L}uczak.
\newblock Biased positional games for which random strategies are nearly
  optimal.
\newblock {\em Combinatorica}, 20(4):477--488, 2000.

\bibitem{Berlekamp68}
E.~Berlekamp.
\newblock A construction for partitions which avoid long arithmetic
  progressions.
\newblock {\em Canadian Mathematical Bulletin}, 11:409--414, 1968.

\bibitem{CE78}
V.~Chv{\'a}tal and P.~Erd\H{o}s.
\newblock Biased positional games.
\newblock {\em Annals of Discrete Mathematics}, 2:221--228, 1978.

\bibitem{CG10}
D.~Conlon and W.~T. Gowers.
\newblock Combinatorial theorems in sparse random sets.
\newblock {\em Annals of Mathematics}, (184(2)):367--454, 2016.

\bibitem{ES73}
P.~Erd{\H o}s and J.~L. Selfridge.
\newblock On a combinatorial game.
\newblock {\em Journal of Combinatorial Theory, Series A}, 14(3):298--301,
  1973.

\bibitem{GebauerSzabo}
H.~Gebauer and T.~Szab{\'o}.
\newblock Asymptotic random graph intuition for the biased connectivity game.
\newblock {\em Random Structures \& Algorithms}, 35(4):431--443, 2009.

\bibitem{Gow2001}
W.~T. Gowers.
\newblock A new proof of {S}zemer\'edi's theorem.
\newblock {\em Geometric and Functional Analysis}, 11(3):465--588, 2001.

\bibitem{GRS90}
R.~Graham, B.~Rothschild, and J.~Spencer.
\newblock {\em {Ramsey Theory}}.
\newblock Wiley, second edition, 1990.

\bibitem{HST17}
R.~Hancock, K.~Staden, and A.~Treglown.
\newblock Independent sets in hypergraphs and {R}amsey properties of graphs and
  the integers.
\newblock Available online at {\tt arXiv:1701.04754}.

\bibitem{JTR90-paper}
S.~Janson, T.~{\L}uczak, and A.~Ruci\'nski.
\newblock An exponential bound for the probability of nonexistence of a
  specified subgraph in a random graph.
\newblock In {\em Random graphs '87 ({P}ozna\'n, 1987)}, pages 73--87. Wiley,
  Chichester, 1990.

\bibitem{JLR00}
S.~Janson, T.~{\L}uczak, and A.~Ruci{\'n}ski.
\newblock {\em Random Graphs}.
\newblock Wiley-Interscience series in Discrete Mathematics and Optimization.
  John Wiley \& Sons, Wiley-Interscience, New York, 2000.

\bibitem{JR11}
S.~Janson and A.~Ruci{\'n}ski.
\newblock Upper tails for counting objects in randomly induced subhypergraphs
  and rooted random graphs.
\newblock {\em Arkiv f{\"u}r Matematik}, 49(1):79--96, 2011.

\bibitem{Krivelevich}
M.~Krivelevich.
\newblock The critical bias for the {H}amiltonicity game is $(1+o(1))n /
  \ln(n)$.
\newblock {\em Journal of the American Mathematical Society}, 24(1):125--131,
  2011.

\bibitem{RR97}
V.~R{{\"o}}dl and A.~Ruci{{\'n}}ski.
\newblock Rado partition theorem for random subsets of integers.
\newblock {\em Proceedings of the London Mathematical Society (3)},
  74(3):481--502, 1997.

\bibitem{RSZ15}
J.~Ru{\'e}, C.~Spiegel, and A.~Zumalac{\'a}rregui.
\newblock Threshold functions and poisson convergence for systems of equations
  in random sets.
\newblock {\em Mathematische Zeitschrift}, 288(1):333--360, 2018.

\bibitem{Rz93}
I.~Z. Ruzsa.
\newblock Solving a linear equation in a set of integers {I}.
\newblock {\em Acta Arithmetica}, 65(3):259--282, 1993.

\bibitem{Sch12}
M.~Schacht.
\newblock Extremal results for random discrete structures.
\newblock {\em Annals of Mathematics}, (184(2)):333--365, 2016.

\bibitem{Sp16}
C.~Spiegel.
\newblock A note on sparse supersatursaturation and extremal results for linear
  homogeneous systems.
\newblock {\em Electronic Journal of Combinatorics}, 24(3):Research Paper 38,
  2017.

\bibitem{SudakovVu}
B.~Sudakov and V.~H. Vu.
\newblock Local resilience of graphs.
\newblock {\em Random Structures \& Algorithms}, 33(4):409--433, 2008.

\bibitem{ZoltanSzabo}
Z.~Szab{\'o}.
\newblock An application of {L}ov{\'a}sz' local lemma -- a new lower bound for
  the van der {W}aerden number.
\newblock {\em Random Structures \& Algorithms}, 1(3):343--360, 1990.

\bibitem{vdW27}
B.~L. Van~der Waerden.
\newblock {Beweis einer Baudetschen Vermutung}.
\newblock {\em Nieuw Arch. Wisk}, 15(2):212--216, 1927.

\end{thebibliography}
\bibliographystyle{abbrv}

\section*{Appendix} \label{sec:remarks}

Recall that the \emph{median} $\median$ of a discrete random variable $X$ satisfies $\PP{X \leq \median} \geq 1/2$ as well as $\PP{X \geq \median} \geq 1/2$. Furthermore recall that any median of the binomial distribution $\mB(n,p)$ lies between $\floor{np}$ and $\ceil{np}$, that is $\floor{np} \leq \median \left( \mB(n,p) \right) \leq \ceil{np}$. A family of subsets $\mP \subseteq 2^{[n]}$  is called \emph{monotone decreasing} if $A \subseteq B$ and $B \in \mP$ implies $A \in \mP$. It is called \emph{monotone increasing} if its complement in $2^{[n]}$ is monotone decreasing. As usual one identifies properties of subsets of $[n]$ with the corresponding family of subsets having the property. The purpose of this appendix is to prove the following lemma:
\begin{lemma}
	Let $X \sim \mB(n,p)$ and let $\mP$ be a monotone decreasing family of subsets of $[n]$.Then there exists a constant $C>0$ such that if $\sqrt{np(1-p)} > C $, then  $\PP{[n]_{\floor{np}} \in \mP} \leq 3 \, \PP{[n]_p \in \mP}$.
\end{lemma}
%

\begin{proof}
	Note that since $\mP$ is monotone decreasing, we have $\PP{[n]_K \in \mP} \geq \PP{[n]_L \in \mP}$ whenever $K \leq L$. Thus
	\begin{align*}
		\PP{[n]_p \in \mP} & = \sum_{M=0}^n \PP{[n]_p \in \mP \, | \, |[n]_p| = M} \, \PP{|[n]_p| = M} = \sum_{M=0}^n \PP{[n]_{M} \in \mP} \, \PP{|[n]_p| = M} \\
		& \geq \sum_{M=0}^{\floor{np}} \PP{[n]_{M} \in \mP} \, \PP{|[n]_p| = M} \geq \PP{[n]_{\floor{np}} \in \mP} \, \sum_{M=0}^{\floor{np}} \PP{|[n]_p| = M}.
	\end{align*}
Note that $\sum_{M=0}^{\floor{np}} \PP{|[n]_p| = M} = \PP{X \leq \floor{np}}$. Let $\mu_{1/2}$ be the median of $X$ and assume first that  $\floor{np} \leq \median < \ceil{np}$. Then $\PP{X \leq \floor{np}} = \PP{X \leq \median} \geq \frac{1}{2}$ and hence
\begin{equation*}
	\mathbb{P}([n]_{\floor{np}} \in \mP) \leq 2 \, \PP{[n]_p\in \mP}.
\end{equation*}
It remains to be shown that the assertion follows as well if $\median = \ceil{np}$. Note that
\begin{equation*}
	\PP{X \leq \floor{np}} = \PP{X \leq \ceil{np}} - \PP{X = \ceil{np}} \geq \frac{1}{2} - \PP{X = \ceil{np}}.
\end{equation*}
We will show that  $\PP{X = \ceil{np}} \leq 1/6$ which then implies $\PP{[n]_{\floor{np}} \in \mP} \leq 3 \, \PP{[n]_p \in \mP}$. To do so, we will upper bound the probability that $X= \ceil{np}$ and use the inequalities $\sqrt{2\pi n} \left(\frac{n}{e}\right)^n \leq n! \leq \sqrt{2\pi n}  \left(\frac{n}{e}\right)^n e$ as follows:
\begin{align*}
	\PP{X = \ceil{np}} & = \binom{n}{\ceil{np}} \, p^{\ceil{np}} \, (1-p)^{n-\ceil{np}} = \frac{n! \, p^{\ceil{np}} \, (1-p)^{n-\ceil{np}}}{\ceil{np}! (n -\ceil{np})!}  \\
	& \leq  \frac{\sqrt{n} \; n^n\; e \; p^{\ceil{np}} \, (1-p)^{n-\ceil{np}}}{\sqrt{2 \pi \ceil{np}} \; (\ceil{np})^{\ceil{np}} \; \sqrt{n-\ceil{np}} \; (n-\ceil{np})^{n-\ceil{np}}} \\
	& = \frac{\sqrt{n}}{\sqrt{n - \ceil{np}}} \frac{(np)^{\ceil{np}}}{\ceil{np}^{\ceil{np}}} \frac{(n-np)^{n-\ceil{np}}}{(n-\ceil{np})^{n-\ceil{np}}} \frac{e}{\sqrt{2 \pi \ceil{np}}},
\end{align*}
Clearly we have $(np)^{\ceil{np}}/\ceil{np}^{\ceil{np}} \leq 1$ as well as $(n-np)^{n-\ceil{np}}/(n-\ceil{np})^{n-\ceil{np}} \leq e$. Hence we get
\begin{equation*}
	\PP{X = \ceil{np}} \leq \frac{e^2}{\sqrt{2\pi}} \, \sqrt{ \frac{n}{n - np - 1} \, \frac{1}{np} } \leq \frac{3}{\sqrt{(1-p)np - p}} < \frac{3}{\sqrt{C^2-1}}.
\end{equation*}
Choosing $C > 0$ large enough such that $\PP{X = \ceil{np}} \leq 1/6$ gives the desired property.
\end{proof}

\end{document}